\documentclass[reqno,12pt]{amsart}
\usepackage{etex}
\usepackage{amsthm,amssymb}
\usepackage{epic,eepic}
\usepackage{epsfig}
\usepackage{xypic}
\usepackage{tikz}
\usetikzlibrary{arrows,positioning} 
\tikzset{
    >=stealth',
    pun/.style={
           rectangle,
           rounded corners,
           draw=black, very thick,
           text width=6em,
           minimum height=2em,
           text centered},
    par/.style={
           ->,
           thick,
           shorten <=2pt,
           shorten >=2pt,}
}

\usepackage[arrow]{xy}
\usepackage{young}

\usepackage{latexsym,cite,epsf,graphics}
\usepackage[a4paper,height=20cm,top=26mm,bottom=26mm,left=26mm,right=26mm]{geometry}
\usepackage{graphicx,color}
\usepackage{tabmac1}
\usepackage{hyperref}

\numberwithin{equation}{section}

\newtheorem{thm}{Theorem}[section]
\newtheorem{cor}[thm]{Corollary}
\newtheorem{lem}[thm]{Lemma}
\newtheorem{prop}[thm]{Proposition}
\theoremstyle{definition}
\newtheorem{definition}[thm]{Definition}
\newtheorem{example}[thm]{Example}

\theoremstyle{remark}
\newtheorem{remark}[thm]{Remark}

\def\i{{\mathbf{i}}}
\def\M{{\mathcal{M}}}
\def\sign{{\rm sign}}
\def\mS{{\mathfrak{S}}}
\def\SI{{\mathfrak{S}(I)}}

\def\tx{{\tilde x}}
\def\tcR{{\tilde {\mathcal{R}}}}
\def\tR{{\tilde R}}

\def\F{{\mathcal F}}
\def\tF{\tilde \F}
\def\Z{{\mathbb Z}}
\def\R{{\mathbb R}}
\def\CC{{\mathcal{C}}}
\def\QP{{\mathbb {QP}}}
\def\P{{\mathbb P}}

\def\Q{{\mathcal Q}}

\def\mR{{\mathcal R}}
\def\QQ{{\mathbb Q}}

\def\C{{\mathbb C}}
\def\Z{{\mathbb Z}}

\def\p{{\mathbf{p}}}
\def\wt{{\rm wt}}
\def\a{{\mathfrak{a}}}
\def\b{{\mathfrak{b}}}

\def\Y{{\mathcal Y}}

\def\bx{{\boldsymbol{x}}}
\def\bX{{\boldsymbol{X}}}
\def\boldy{{\boldsymbol{y}}}

\def\x{\mathbf{x}}

\def\bq{\mathbf{q}}
\def\p{\mathbf{p}}
\def\q{\mathbf{q}}
\def\r{\mathbf{r}}
\def\e{\epsilon}
\def\wte{{\rm wt}_\e}
\def\tQ{{\tilde Q}}
\def\asln{\hat {\mathfrak {sl}}_n}

\def\Sym{{\rm Sym}}

\def\comb{{\rm comb}}
\def\id{{\rm id}}

\usepackage{lipsum}

\begin{document}

\baselineskip 16pt

\title[Cluster nature of $R$-matrices]
{On the cluster nature and quantization of geometric $R$-matrices}

\author[Rei Inoue]{Rei Inoue}
\address{Rei Inoue, Department of Mathematics and Informatics,
   Faculty of Science, Chiba University,
   Chiba 263-8522, Japan.}
\email{reiiy@math.s.chiba-u.ac.jp}
\thanks{R.~I. was partially supported by JSPS KAKENHI Grant Number
26400037.}

\author[Thomas Lam]{Thomas Lam}
\address{Thomas Lam, Department of Mathematics, 
University of Michigan, Ann Arbor, MI 48109, USA.}
\email{tfylam@umich.edu}
\thanks{T.L. was partially supported by NSF grants DMS-1160726, DMS-1464693, and a Simons Fellowship.}

\author[Pavlo Pylyavskyy]{Pavlo Pylyavskyy} 
\address{Pavlo Pylyavskyy, School of Mathematics, University of Minnesota, 
Minneapolis, MN 55414, USA.}
\email{ppylyavs@umn.edu}
\thanks{P.~P. was partially supported by NSF grants DMS-1148634, DMS-1351590, and Sloan Fellowship.}


\date{Nov 7, 2017}
\subjclass[2010]{13F60, 17B37}

\keywords{}

\begin{abstract}
We define cluster $R$-matrices as sequences of mutations in triangular grid quivers on a cylinder, and show that the affine geometric $R$-matrix of symmetric power representations for the quantum affine algebra $U_q^\prime(\asln)$
can be obtained from our cluster $R$-matrix.  
A quantization of the affine geometric $R$-matrix is defined,  compatible with the cluster structure.  We construct invariants of the quantum affine geometric $R$-matrix as quantum loop symmetric functions.
\end{abstract}

\maketitle

\setcounter{tocdepth}{1}
\tableofcontents

\section{Introduction}

The $R$-matrix of a quantum group, or that of a representation of a quantum group, plays a vital role in the connection between representation theory and the theory of integrable systems.  
The transformations appearing in this paper arise from $R$-matrices of Kirillov-Reshitikhin modules of the quantum affine algebra $U_q^\prime(\asln)$,
from which {\it affine combinatorial $R$-matrices} are obtained as crystal limits, see \cite{KKMMNN}.   
The combinatorial $R$-matrix is a bijection $R_{\comb}\colon B \otimes B' \to B' \otimes B$ between tensor products of Kirillov-Reshetikhin crystals $B$ and $B'$.   In this paper, we focus on the case where $B = B_\ell$ (resp. $B' = B_{\ell'}$) is the crystal graph of the symmetric tensor module $\Sym^\ell \C^n$ (resp. $\Sym^{\ell'} \C^n$).  

In this case, the bijection $R_{\comb}$ can be interpreted as the restriction of a piecewise-linear map $\R^n \times \R^n \to \R^n \times \R^n$ to a finite subset of lattice points.  This piecewise-linear map is the tropicalization of a rational morphism $R\colon (\C^*)^n \times (\C^*)^n  \to (\C^*)^n \times (\C^*)^n $ called the {\it geometric $R$-matrix} \cite{KNY,KNO,Et,TP}, which plays an important role in the theory of affine geometric crystals.  We let $\p = (p_1,\ldots,p_n)$ and $\q = (q_1,\ldots,q_n)$ be coordinates on $(\C^*)^n \times (\C^*)^n$.  Then the geometric $R$-matrix is an isomorphism $R\colon \QQ(\p,\q) \to \QQ(\p,\q)$ of fields.  It satisfies the Yang-Baxter relation, giving an action of the symmetric group $\mS_m$ on $\QQ(\q_1,\q_2,\ldots,\q_m)$ by rational transformations, where $\q_j = (q_{j,1},\ldots,q_{j,n})$. 

\subsection{Quantum geometric $R$-matrix}
Our first main result is the construction of a quantization of the rational map $R$, called the {\it quantum geometric $R$-matrix} $R^\e$.  The quantum geometric $R$-matrix is an isomorphism $R^\e\colon \QQ_\e \langle \p,\q \rangle \to \QQ_\e \langle \p,\q \rangle$ between the skew fraction fields of certain quantum tori.  
Here $\e$ is an indeterminate, commuting with all $p_i$ and $q_i$.
Theorem \ref{thm:q-qYBR} states that, like the classical, combinatorial, and geometric $R$-matrices, the quantum geometric $R$-matrix also satisfies a Yang-Baxter relation.  We may summarize the four types of $R$-matrix transformations as follows.

\bigskip

\begin{center}
\begin{tabular}{|c|c|c|}
\hline
 & Transformation & Acts on \\
\hline
(quantum) $R$-matrix & $\C(q)$-linear & $U_q^\prime(\asln)$-representations $V \otimes V'$ \\
\hline
combinatorial $R$-matrix & piecewise linear & $\asln$-crystals $B \otimes B'$ \\
\hline
geometric $R$-matrix & rational map & affine geometric crystals $X \times X'$ \\
\hline
quantum geometric $R$-matrix & skew field map & quantum geometric crystals? \\
\hline
\end{tabular}
\end{center}

\bigskip

The quantum geometric $R$-matrix generates an action of the symmetric group $\mS_m$ on a skew field $\mathcal{Q}_\e := \QQ_\e \langle \q_1,\ldots,\q_m \rangle$.  We give a `symmetric function' description of certain elements in the invariant subfield $\Q_\e^{\mS_m}$.  Namely, we define  quantum loop elementary symmetric functions (Corollary~\ref{cor:elem}), quantum loop Schur functions (Theorem \ref{thm:schur}) and quantum cylindric loop Schur functions (Theorem \ref{thm:q-cylindric-loop-schur}), and prove that they are invariants.  This generalizes results from \cite{Yam01,TP,LPS}.  

Quantum loop Schur functions and quantum cylindric loop Schur functions have two closely related descriptions: (1) as the (noncommutative) generating function of certain semistandard Young tableaux or semistandard cylindric tableaux, and (2) as the (noncommutative) generating function of families of {\it highway paths} in a network $N_{n,m}$ on the cylinder, previously studied in \cite{LP}.  We quantize an argument from \cite{LP} to show that the quantum geometric $R$-matrix can be obtained as a composition of certain local transformations of the cylindric network $N_{n,m}$.

\subsection{Cluster $R$-matrix}
We discovered the quantum geometric $R$-matrix after realizing the geometric $R$-matrix as a composition of cluster mutations for a certain `triangular grid' quiver $Q_{n,m}$ 
on a cylinder.  Our second main result is the construction of this {\em cluster $R$-matrix}.  The cylindric quiver $Q_{n,m}$ and the cylindric network $N_{n,m}$ are related in a similar way to how the quiver of the cluster algebra for a 
Bruhat cell is related to the wiring diagram of a reduced word \cite{FZ3}.

The cluster $R$-matrix is defined to be the composition of $2n-2$ cluster mutations of $Q_{n,m}$ which (rather non-trivally) sends $Q_{n,m}$ back to itself, and furthermore satisfies a Yang-Baxter relation (Theorem \ref{thm:Rbraidsimple}).
Further we consider a decorated quiver $\tilde Q_{n,m}'$ of $Q_{n,m}$ with 
frozen variables, and the natural coordinates $\q_i$ of the geometric $R$-matrix are related to the cluster variables $x$ and frozen variables $X$ of the corresponding cluster algebra by a relation of the form
\begin{equation}\label{eq:xxxx}
q = \frac{x x'}{x'' x'''} X.
\end{equation}
This is reminiscent of $\tau$-function substitutions in the theory of integrable systems, and of formulae occurring in the chamber ansatz \cite{BFZ}.

Broadly speaking, there are two notions of quantization for cluster algebras.  Berenstein and Zelevinsky \cite{BZ} have defined quantum cluster algebras that quantize cluster mutation. 
This notion however depends on additional choices.  Fock and Goncharov\cite{FockGoncharov03} have defined a canonical quantization of mutation of cluster $y$-seeds, or equivalently, of the coefficient dynamics.  Working with the same sequence of mutations of $Q_{n,m}$, but this time with quantum cluster $y$-seeds, we construct the {\em quantum cluster $R$-matrix}.
We show that the quantum cluster $R$-matrix satisfies a Yang-Baxter relation (Theorem \ref{thm:q-YBR}) and that it is compatible with the quantum geometric $R$-matrix (Theorem \ref{thm:psi-RR}) via an embedding $\phi\colon \Y'_\e \to \Q_\e$, where $\Y'_\e$ is a subfield of the quotient skew field of the quantum torus of a $y$-seed.

The following diagram is an overview of the relations between the different $R$-matrices.
We denote our results in this paper in red. 
The arrow $(\ast)$ is established by the `$\tau$-function substitution' \eqref{eq:xxxx}.  The arrow $(\ast \ast)$ is established by the embedding of skew fields $\phi$.

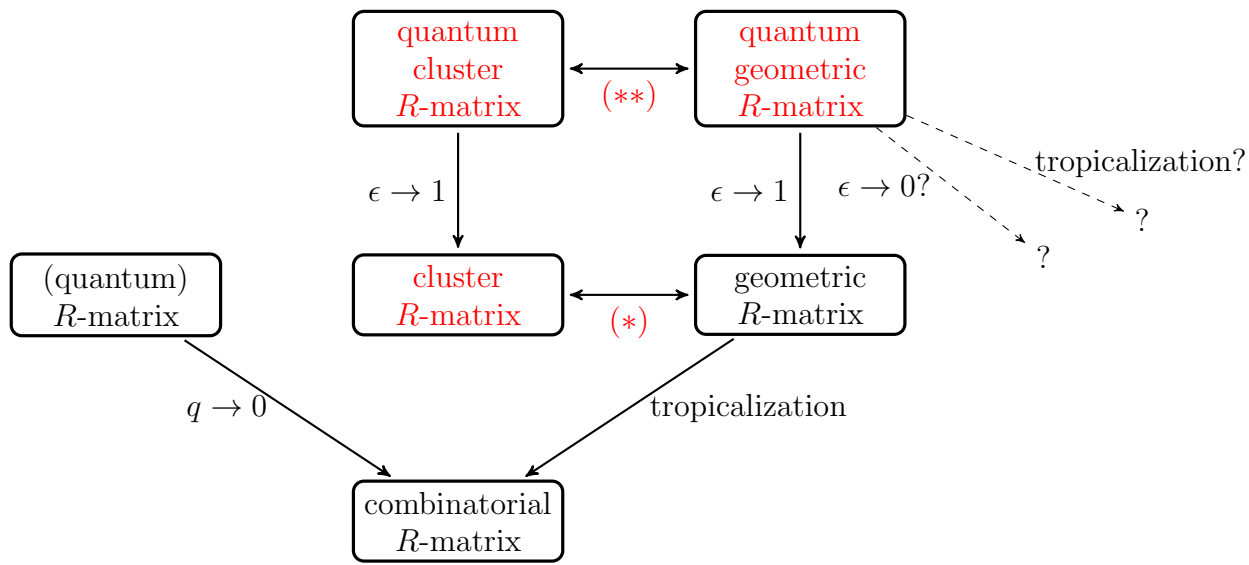
\begin{center}
\begin{tikzpicture}[node distance=1cm, auto,]
 \node[pun]  at (-8,0) (comb) {combinatorial $R$-matrix};
 \node[pun] at (-3.5,3)  (geom) {geometric $R$-matrix} edge[par] node[right] {\;tropicalization} (comb);
 \node[pun] at (-12.5,3) (R) {(quantum) $R$-matrix} edge[par] node[left] {$q \to 0\;$} (comb); 
  \node at (1,4) (trop) {?};
 \node at (-0.3,3.5) (crys) {?};
 \node[pun] at (-3.5,6) (quantum) {\color{red} quantum geometric $R$-matrix} edge[par] node[left] {$\epsilon \to 1$} (geom) edge[dashed,->] node[right] {\;tropicalization?} (trop) edge[dashed,->] node[left] {$\epsilon \to 0$?\;} (crys);
 \node[pun] at (-8,3) (cluster) {\color{red} cluster $R$-matrix} edge[par,<->] node[below]{\color{red}$(\ast)$} (geom);
 \node[pun] at (-8,6) (qcluster) {\color{red} quantum cluster $R$-matrix} edge[par] node[left] {$\epsilon \to 1$} (cluster) edge[par,<->] node[below] {\color{red}$(\ast \ast)$} (quantum);
\end{tikzpicture}

\end{center}

{\bf Outline.}
In \S\ref{sec:clust} we recall the necessary background on cluster algebras and their properties. 
In \S\ref{sec:clR} we introduce the cluster $R$-matrix
as a sequence of mutations on a triangular grid quiver on a cylinder.
First we study the case of a simple quiver without frozen vertices,
then we consider the case of the decorated quiver with frozen vertices.
In both cases we prove that the $R$-matrices generate the symmetric group 
action on the field of cluster variables.
In \S\ref{sec:geom} we show how the geometric $R$-matrix is induced by the cluster $R$-matrix. 
In \S\ref{sec:qgeom}, we define a skew field $\mathcal{Q}_\e$ 
and the quantum geometric $R$-matrices acting on it. 
We state the result that the quantum geometric $R$-matrices generate a symmetric group action on $\mathcal{Q}_\e$.
In \S\ref{sec:inv} we define some distinguished invariants of the quantum geometric $R$-matrix. We use a network model on a cylinder to prove the invariance. 
In \S\ref{sec:qclus} we define the quantum cluster $R$-matrices, using Fock-Goncharov quantization of cluster $y$-seeds. 
The quantum cluster $R$-matrices generate the symmetric group action on the skew field of quantum $y$-variables.
We then show that
the quantum cluster $R$-matrix is compatible with the quantum geometric $R$-matrix. In particular, we use the Yang-Baxter relation for the former to deduce the Yang-Baxter relation for the latter.
In \S\ref{sec:proof} we collect the proofs of some technical results in \S\ref{sec:clR}. 

\medskip
\noindent
{\bf Acknowledgments.}  
We thank the anonymous referees for helpful comments.

\section{Cluster algebra preliminaries} \label{sec:clust}
\subsection{Seeds and mutation}
We recall the notion of seeds and mutation of cluster algebras from \cite{FominZelev02}.  Our cluster algebras will be skew-symmetric.

Fix a finite set $I$, and let $N$ be its cardinality. 
Let $(B,\boldsymbol{x})$ denote a seed,  
where $\boldsymbol{x}=(x_i)_{i \in I}$ is a collection of cluster variables
and $B=(b_{ij})_{i,j \in I}$ is an $N\times N$ skew-symmetric integral matrix called the exchange matrix.  The variables $x_i$ are assumed to be algebraically independent elements of some ambient field.  For $k \in I$, 
the mutation $\mu_k$ of a seed $(B,\boldsymbol{x})$ is defined to be 	
$(\tilde B, \tilde{\boldsymbol{x}}) = \mu_k(B,\boldsymbol{x})$, where
\begin{align}
&  \tilde x_i
  =
  \begin{cases}
    x_i & i \neq k,
    \\
    \displaystyle{
    \frac{\prod_{j: b_{jk}>0} x_j^{~b_{jk}}
      + \prod_{j: b_{jk} < 0} x_j^{~-b_{jk}}
    }{x_k}} 
    & i = k,
  \end{cases}
  \\
  \label{eq:B-mutation}
&  \tilde b_{ij}
  =
  \begin{cases}
    - b_{ij}
    & \text{$i=k$ or $j=k$},
    \\
    \displaystyle{
    b_{ij}
    + \frac{
      \bigl| b_{ik} \bigr| \, b_{kj}
      +
      b_{ik} \, \bigl| b_{kj} \bigr|
    }{2}}
    & \text{otherwise.}
  \end{cases}
\end{align}
Mutation is an involutive operation, and $\mu_i$ and $\mu_j$ commute when $b_{ij}=0$.  For an exchange matrix $B$, 
we have a quiver $Q = Q(B)$ with vertices identified with the set $I$, without $1$-cycles and $2$-loops, satisfying
\begin{align*}
  b_{ij}
  =
  \# \left\{ \text{arrows from $i$ to $j$} \right\}
  -
  \# \left\{ \text{arrows from $j$ to $i$} \right\},
\end{align*}
for $i,j \in I$.  The mutation $\mu_k(Q) = Q(\mu_k(B))$ of a quiver $Q = Q(B)$ is obtained by the following procedure:
\begin{enumerate}
\item
for every pair of directed edges $i \to k$ and $k \to j$, we add a directed edge $i \to j$;
\item
reverse all directed edges going into or coming out of $k$;
\item
remove any $2$-cycles in pairs $\{i \to j, j \to i\}$.
\end{enumerate}

Often one is in the situation where some vertices of a quiver $Q$ (resp. some cluster variables) are never mutated.  We call such vertices (resp. variables) frozen, and the other vertices (resp. variables) are called mutable.  In practice, edges between frozen vertices can be disregarded (resp. columns of the exchange matrix $B$ indexed by frozen vertices can be ignored).

\subsection{Seeds with coefficients and $y$-seeds}

We recall the notion of seeds with coefficients and the $y$-seeds
from \cite{CA4}.

Let $(\mathbb{P},\cdot,\oplus)$ be a semifield, that is,
an abelian multiplicative group with an addition $\oplus$
which is commutative, associative, and distributive
with respect to the multiplication $\cdot$ in $\mathbb{P}$.
Let $\QP$ denote the fraction field of the group
ring $\Z\P$.  Let $\F$ denote a rational function field over $\QP$ generated by some algebraically independent elements.  

A seed with coefficients $(B,\bx,\boldy)$ consists of a $N \times N$ skew-symmetric integral matrix $B$, a collection $\bx=(x_i)_{i \in I}$ of algebraically independent elements in $\F$, and a collection $\boldy = (y_i)_{i \in I}$ of elements in $\P$.
Each $y_i$ is called a coefficient, or a $y$-variable.
For $k \in I$, the mutation $\mu_k(B,\bx,\boldy) = (\tilde B,\tilde \bx,\tilde \boldy)$ is defined as follows.  The matrix $\tilde B$ is defined as before, and we have
\begin{align}
&\tilde x_i = \begin{cases} x_i & i \neq k, \\
  \displaystyle{
    \frac{y_k\, \prod_{j: b_{jk}>0} x_j^{~b_{jk}}
      + \prod_{j: b_{jk} < 0} x_j^{~-b_{jk}}
    }{(1 \oplus y_k) x_k}} 
    & i = k,
\end{cases}
\\[1mm] 
\label{eq:classicalY}
&\tilde y_i
  =
  \begin{cases}
    y_k^{~-1} & i = k,
    \\
    y_i \, \left( 1 \oplus y_k^{~-1} \right)^{-b_{ki}} 
    & i \neq k, ~b_{ki} \geq 0,
    \\
    y_i \left( 1 \oplus y_k \right)^{-b_{ki}} 
    & i \neq k,~b_{ki} \leq 0.
  \end{cases}
\end{align}
A $y$-seed $(B,\boldy)$ and its mutations are
defined by ignoring the cluster variables $\bx$ of the seed $(B,\bx,\boldy)$.

Two particular semifields will be important to us: the universal semifield $\mathbb{P}_{\rm univ}(\boldy^0)$
and the tropical semifield $\mathbb{P}_{\rm trop}(\boldy^0)$,
generated by a set $\boldy^0 = (y_i^0)_{i \in I}$ of initial $y$-variables.
The universal semifield $\mathbb{P}_{\rm univ}(\boldy^0)$ is the subset of 
all rational functions $f$ in $y_i^0 ~(i \in I)$ such that 
$f \in \mathbb{P}_{\rm univ}(\boldy^0)$ has a subtraction-free rational expression, endowed with the ordinary multiplication and 
addition `$+$' of rational functions.
The tropical semifield $\mathbb{P}_{\rm trop}(\boldy^0)$ is an abelian multiplicative group on $y_i^0 ~(i \in I)$ endowed with the addition $\oplus$ defined by
$$
  \prod_{i \in I} (y_i^0)^{n_i} \oplus \prod_{i \in I} (y_i^0)^{m_i}
  = \prod_{i \in I} (y_i^0)^{\min(n_i,m_i)},
$$  
for $n_i, m_i \in \Z$.

We define a surjection $\pi \colon \mathbb{P}_{\rm univ}(\boldy^0) \to
\mathbb{P}_{\rm trop}(\boldy^0)$ given by
$$
  f \mapsto [f] := \prod_{i \in I} (y_i^0)^{n_i}.
$$ 
Here $[f]$ is the called the `principal coefficient' of $f$, determined uniquely by writing
$$
  f = \prod_{i \in I} (y_i^0)^{n_i} \cdot \frac{f_1}{f_2}
$$
where $f_1$ and $f_2$ are polynomials of the $y_i^0$ with non-zero constant terms.  It is easy to check that $\pi$ is a homomorphism of semifields.

\subsection{Fock-Goncharov quantization}\label{ssec:FG}

We introduce the quantization of the $y$-seeds, following
\cite{FockGoncharov03}.

Let $(B,\boldy)$ be a quantum $y$-seed, where
$B=(b_{ij})_{i,j \in I}$ be an $N$ by $N$ skew-symmetric integral matrix as
in the classical case, and $\boldy = (y_i)_{i \in I}$ generate 
the quantum torus $\Z_\e[\boldy]$, the noncommutative $\Z[\e,\e^{-1}]$- algebra with relations given by
$$
y_i y_j = \e^{2 b_{ji}} y_j y_i.
$$
Here $\e$ is an indeterminate belonging to the center of the algebra.
Let $\mathcal{Y}_\e(B)$ be the skew fraction field of $\Z_\e[\boldy]$.
Let $\mu_k^\e$ be the quantum mutation of a $y$-seed $(B,\boldy)$, where 
$(\tilde B,\tilde{\boldsymbol{y}}) = \mu_k^\e (B,\boldsymbol{y})$ is given by
%
%
\begin{equation}\label{eq:quantumY}
  \tilde y_i
  =
  \begin{cases}
    y_k^{-1}  & \mbox{if $i = k$,}
    \\
    y_i \,
    \prod_{m=1}^{b_{ki}}
    \left( 1+\e^{2m-1} \, y_k^{-1} \right)^{-1} 
    & \mbox{if $i \neq k, ~b_{ki} \geq 0$,}
    \\
    y_i \,
    \prod_{m=1}^{-b_{ki}}
    \left( 1+ \e^{2m-1} \,  y_k \right) 
    & \mbox{if $i \neq k, ~b_{ki} \leq 0$,}
  \end{cases}
\end{equation}
and $\tilde B$ is the same as in the classical case \eqref{eq:B-mutation}.
Let $\mathcal{Y}_{\e}(\tilde B)$  be the skew fraction field generated by the $\tilde y_i$ with the $\e$-commutativity determined by $\tilde B$.
It is known that $\mu_k^\e$ induces a morphism of the skew fraction fields
$\mathcal{Y}_\e(\tilde B) \to \mathcal{Y}_\e(B)$.
Similarly to the classical case,  $\mu_i^\e$ is involutive, and 
$\mu_i^\e$ and $\mu_j^\e$ commute when $b_{ij} = 0$. 
By setting $\e = 1$, \eqref{eq:quantumY} reduces to \eqref{eq:classicalY}
of the case $\mathbb{P} = \mathbb{P}_{\rm univ}(\boldy^0)$.

\subsection{Periods of cluster mutations}\label{ssec:periods}
The permutation group $\SI$ naturally acts on $y$-seeds by permuting the variables and the rows and columns of the exchange matrix.
For $\sigma \in \SI$, a {\it $\sigma$-period} of a $y$-seed is a sequence of mutations that sends a seed back to itself up to $\sigma$; see \cite{Nak11} for more details.  
The following theorem gives an important and useful relation among
the classical, tropical, and quantum cluster $y$-seeds.
 
\begin{thm}\label{thm:period}
For a positive integer $k$ and an $I$-sequence
${\bf i} := (i_1,i_2,\ldots,i_k) \in I^k$,
define a sequence of mutations 
$\mu_{\bf i} 
:= \mu_{i_k} \circ \mu_{i_{k-1}} \circ \cdots \circ \mu_{i_1}$.  Let $\sigma \in \SI$ be a permutation.
Then the following three statements are equivalent:
\begin{itemize}
\item[(1)]
Assume that the semifield is $\mathbb{P}_{\rm univ}(\boldy)$.
The $I$-sequence $\bf i$ is a $\sigma$-{\it period} of 
the $y$-seed $(B,\boldsymbol{y})$, i.e.,   
$\mu_{\bf i}(B,\boldsymbol{y}) =\sigma(B,\boldsymbol{y})$. 

\item[(2)]
Assume that the semifield is $\mathbb{P}_{\rm trop}(\boldy)$.
The $I$-sequence $\bf i$ is a $\sigma$-{\it period} of 
the tropical $y$-seed $(B,\boldsymbol{y})$, i.e.,   
$\mu_{\bf i}(B,\boldsymbol{y})= \sigma(B,\boldsymbol{y})$. 

\item[(3)]
The $I$-sequence $\bf i$ is a $\sigma$-{\it period} of 
the quantum $y$-seed $(B,\boldsymbol{y})$, i.e.,   
$\mu_{\bf i}^\e (B,\boldsymbol{y}) = \sigma(B,\boldsymbol{y})$. 

\end{itemize}
\end{thm}

\begin{proof}
The claim $(1) \Leftrightarrow (3)$ is proved in 
\cite[Lemma 2.22]{FockGonc09a} when the matrix $B$ is non-degenerate, and
generalized to degenerate $B$ in \cite[Proposition 3.4]{KN}.
The claim $(2) \Rightarrow (1)$ is included in \cite[Theorem 5.1]{IIKKN}. 
The claim $(1) \Rightarrow (2)$ is deduced from the map $\pi$.
\end{proof} 

We will also use the following result from \cite[Theorem 5.1]{IIKKN}.

\begin{thm}\label{thm:periodxy}
Let $\sigma \in \SI$ be a permutation and $\i$ be an $I$-sequence as in Theorem \ref{thm:period}.  Suppose that $\P$ is the semifield $\mathbb{P}_{\rm trop}(\boldy)$, and $\i$ is a $\sigma$-period of the seed $(B,\boldy)$.  Then $\i$ is a $\sigma$-period of the seed $(B,\bx,\boldy)$, i.e., $\mu_{\bf i}(B,\bx,\boldy)= \sigma(B,\bx,\boldy)$.  
\end{thm}

\section{Cluster $R$-matrix} \label{sec:clR}

\subsection{Triangular grid quivers on a cylinder}

We consider triangular grid quivers on a cylinder, as in Figure \ref{fig:Rcl1}.  Let $Q_{n,m}$ denote such a quiver with $n$ vertices on each of the two boundary circles and $m+1$ parallel vertical cycles $M_0,M_1,\ldots,M_m$ that go around the cylinder, each with $n$ vertices.  
See Figure \ref{fig:Rcl1} for the example of $Q_{5,5}$,
where the $6$ vertical cycles $M_0,\ldots,M_5$ are arranged from left to right.

\begin{figure}[ht]
\scalebox{0.8}{\input{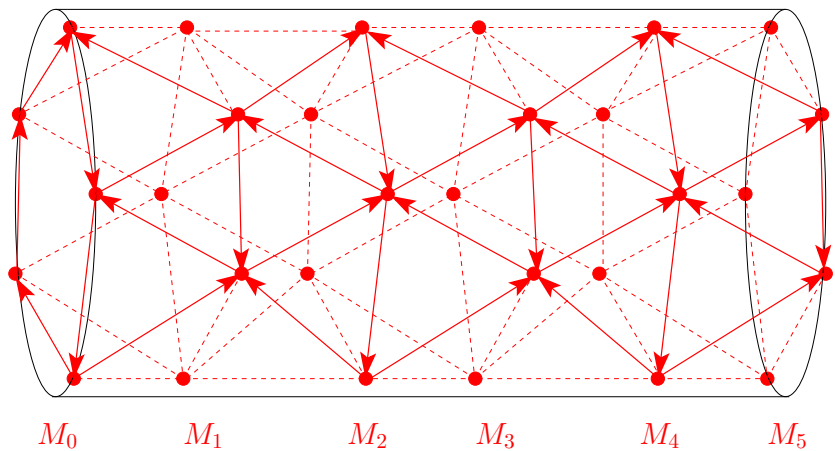}}
    \caption{The quiver $Q_{5,5}$. 
    }
    \label{fig:Rcl1}
\end{figure}

\subsection{Definition of the simple cluster $R$-matrix}\label{subsec:x-R}

Let $(\x^-,\x,\x^+)$ be a tuple of $3n$-variables $x_i^-,x_i,x_i^+$ where $i \in \Z/n\Z$.  The {\it simple cluster $R$-matrix} is the rational transformation $R_\x\colon \QQ(\x^-,\x,\x^+) \to \QQ(\x^-,\x,\x^+)$ given by 
\begin{equation}\label{eq:Rx}
R_\x(x_i) =  
            \displaystyle{
      \frac{\sum_{j=1}^n x_{j+1}^- 
            \left(\prod_{\ell = j+2}^{j-1} x_{\ell} \right) x_j^+}
           {\prod_{j \not = i} x_j}
      }, \qquad R_\x(x^-_i) = x_i^-, \qquad R_\x(x_i^+) = x_i^+.
\end{equation}

For example, if $n = 4$, we have
$$R_\x(x_1) = \frac{x_2^- x_3 x_4 x_1^+ + x_3^- x_4 x_1 x_2^+ + x_4^- x_1 x_2 x_3^+ + x_1^- x_2 x_3 x_4^+}{x_2 x_3 x_4}.$$

Let $Q$ be a quiver.  For a vertex $v \in Q$, we have a cluster variable $x_v$.  Let $\F(Q) = \QQ(x_v \mid v \in Q)$ be the field generated by all the cluster variables indexed by vertices of $Q$.  In this section we take $\F =\F(Q_{n,m})$.  Fix $1 \leq i \leq m-1$.  Let $(M^-,M,M^+)$ denote the three cycles $M_{i-1},M_i,M_{i+1}$.   Let the vertices of $M$ (resp. $M^-$, $M^+$) be denoted $i \in \Z/n\Z$ (resp. $i^-$, $i^+$).  The edges to and from $i$ are given by 
\begin{equation}\label{eq:arrows} i \rightarrow i+1, \qquad  i \rightarrow i^-, \qquad i \rightarrow i-1^+, \qquad i \leftarrow i-1,  \qquad i \leftarrow i+1^-, \qquad i \leftarrow i^+.
\end{equation}
Denote the corresponding cluster variables on $M^-, M, M^+$ by $x_i^-, x_i, x_i^+$ respectively.  We then have a rational transformation $R_M = R_{\x}\colon \F \to \F$ acting on the cluster variables $(\x^-,\x,\x^+)$ and fixing all cluster variables not contained in $M^- \cup M \cup M^+$.

\begin{figure}[ht]
\scalebox{0.8}{\input{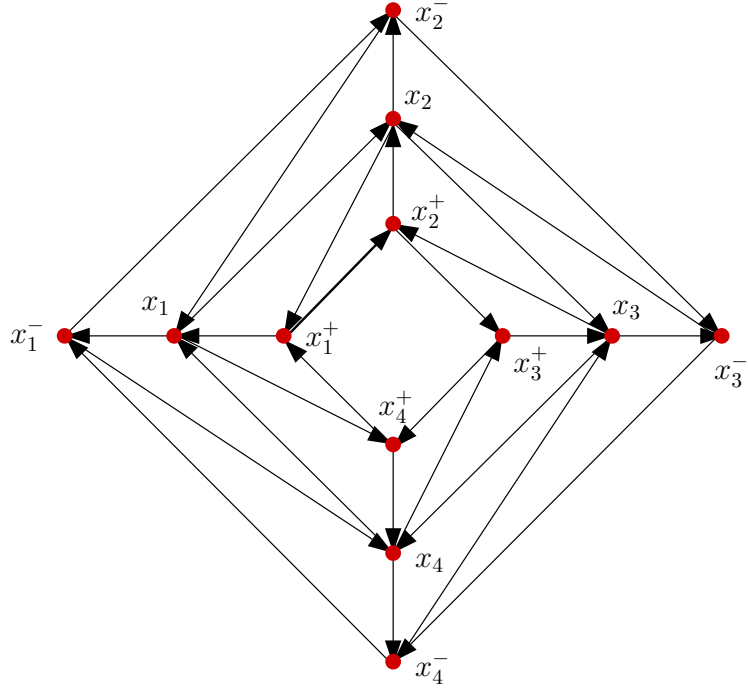}}
    \caption{Adjacent vertical closed cycles $M^-, M, M^+$.}
    \label{fig:clR2}
\end{figure}

Figure \ref{fig:clR2} shows all edges of $Q_{n,m}$ incident with a closed vertical cycle $M$, where $n = 4$.  
Denote by $\mu_i$ the mutation at the vertex $i$ of $M$, and $s_{i,j}$ the permutation of the vertices $i$ and $j$ on $M$.
For $j \in \Z/n \Z$, we define a sequence of $2n-2$ mutations and a permutation by 
\begin{align}\label{eq:R-op}
\mathcal{R}_{M,j} 
:= s_{j-2,j-1} \circ 
(\mu_{j} \circ \mu_{j+1} \cdots \circ \mu_{j-4} \circ \mu_{j-3}) \circ 
(\mu_{j-1} \circ \mu_{j-2} \circ \cdots \circ \mu_{j+1} \circ \mu_j).
\end{align}
The operation $\mathcal{R}_{M,j}$ acts both on the quiver $Q_{n,m}$, sending it to $\mathcal{R}_{M,j}(Q_{n,m})$, and acts as a rational transformation of $\F$. 

\begin{thm}\label{thm:Rcluster}
For any $j \in \Z/n\Z$, we have
$\mathcal{R}_{M,j}(Q_{n,m}, \bx_{Q_{n,m}}) = (Q_{n,m}, R_M(\bx_{Q_{n,m}}))$.
%
\end{thm}
%


\begin{thm} \label{thm:Rbraidsimple}
The $m-1$ transformations $R_{M_1},R_{M_2},\ldots,R_{M_{m-1}}\colon \F \to \F$ satisfy $R_{M_i} R_{M_{i+1}} R_{M_i} = R_{M_{i+1}} R_{M_i} R_{M_{i+1}}$ and $R_{M_i}^2 = \id$.  
Furthermore, for $i,j$ such that $|i-j|>1$ we have
$R_{M_i} R_{M_j} = R_{M_j} R_{M_i}$.
Thus $R_{M_i}$ generates an action of the symmetric group $\mS_m$ on $\F$.
\end{thm}

See \S \ref{subsec:Rcluster} and \S \ref{subsec:Rbraidsimple} for the proofs.

\subsection{Definition of the cluster $R$-matrix}
To the $3n$ variables $(\x^-,\x,\x^+)$, we add $5n$ additional variables denoted $X_{j,j+1}, X_{j^+,j}, X_{j,j^-}, X_{j+1^-,j}, X_{j+1,j^+}$, where $j \in \Z/n\Z$.  Denote by $\bX$ the set of these new variables, and by $\QQ(\bx,\bX)$ the rational function field in these variables.
The {\it cluster $R$-matrix} is the rational transformation $\tR_\x\colon \QQ(\bx,\bX) \to \QQ(\bx,\bX)$ given by 
\begin{align}\label{eq:RonXx}
\tR_\x(x_i) &= \frac{\sum_{j=1}^n  x_{j+1}^- \left(\prod_{\ell = j+2}^{j-1} x_{\ell} \right) x_j^+   X_{j,j+1} \left( \prod_{\ell=j+1}^{i-1} X_{\ell^+,\ell} X_{\ell+1^-,\ell}\right) \left(\prod_{\ell=i+1}^j  X_{\ell,\ell^-} X_{\ell,\ell-1^+}\right) }{\prod_{j \not = i} x_j} 
\end{align}
\begin{align*}
\tR_\x(x_i^-) = x_i^-, \qquad \tR_\x(x_i^+) = x_i^+,
\end{align*}
and
\begin{align}\label{eq:RonX}
\begin{split}
&\tR_\x(X_{i^+,i}) = X_{i+1,i+1^-}, \qquad \tR_\x(X_{i+1,i+1^-}) = X_{i^+,i}, \\
&\tR_\x(X_{i ,i-1^+}) = X_{i^-,i-1}, \qquad \tR_\x(X_{i^-,i-1}) = X_{i, i-1^+}, \\
&\tR_\x(X_{i,i+1}) = X_{i,i+1}.
\end{split}
\end{align}
For example, if $n = 4$, we have
\begin{align*}x_2 x_3 x_4 \tR_\x(x_1) = \, & x_2^- x_3 x_4 x_1^+ X_{1,2} X_{2^+,2} X_{3^-,2} X_{3^+,3} X_{4^-,3} X_{4^+,4} X_{1^-,4}\\ &+ x_3^- x_4 x_1 x_2^+ X_{2,3} X_{3^+,3} X_{4^-,3} X_{4^+,4} X_{1^-,4} X_{2,2^-} X_{2,1^+} \\
&+ x_4^- x_1 x_2 x_3^+ X_{3,4} X_{4^+,4} X_{1^-,4} X_{2,2^-} X_{2,1^+} X_{3,3^-} X_{3,2^+} \\
&+ x_1^- x_2 x_3 x_4^+X_{4,1} X_{2,2^-} X_{2,1^+} X_{3,3^-} X_{3,2^+} X_{4,4^-} X_{4,3^+}.
\end{align*}

To each arrow $a \rightarrow a'$ in the quiver $Q_{n,m}$ associate a new (frozen) vertex $v_{aa'}$, connecting it to already existing vertices of the quiver via arrows 
$a' \rightarrow v_{a,a'}$  and $v_{a,a'} \rightarrow a$. Denote the resulting enriched quiver by $\tilde Q_{n,m}$.  Associate to each of the new vertices $v_{a,a'}$ a frozen variable $X_{a,a'}$, and let $\tF = \F(\tilde Q_{n,m})$ be the rational function field generated by all the cluster variables, including the new frozen variables.

As before, let $(M^-,M,M^+)$ denote the three adjacent cycles $M_{i-1},M_i,M_{i+1}$.
The cycle $M$ is connected to vertices from the two adjacent cycles $M^-$ and $M^+$. Let 
$n$ be the length of $M$, and denote the vertices on $M^-,M,M^+$ as $i^-, i, i^+$ respectively, where $i$ takes values $\mod n$ and the arrows around $i$ are now 
$$i \rightarrow i+1, \qquad i \rightarrow i^-, \qquad i \rightarrow i-1^+, \qquad i \leftarrow i-1, \qquad i \leftarrow i+1^-, \qquad i \leftarrow i^+,$$ 
$$i \rightarrow v_{i-1,i}, \qquad i \rightarrow v_{i+1^-, i}, \qquad i \rightarrow v_{i^+, i}, \qquad i \leftarrow v_{i ,i+1}, \qquad i \leftarrow v_{i ,i^-}, \qquad i \leftarrow v_{i, i-1^+}.$$

We then have a rational transformation $\tR_M = \tR_\x\colon \tF \to \tF$ acting on the mutable variables on $M^- \cup M \cup M^+$, and permuting some of the frozen variables; the transformation fixes all the variables not appearing in \eqref{eq:RonXx} or \eqref{eq:RonX}.

\begin{figure}[ht]
\scalebox{0.8}{\input{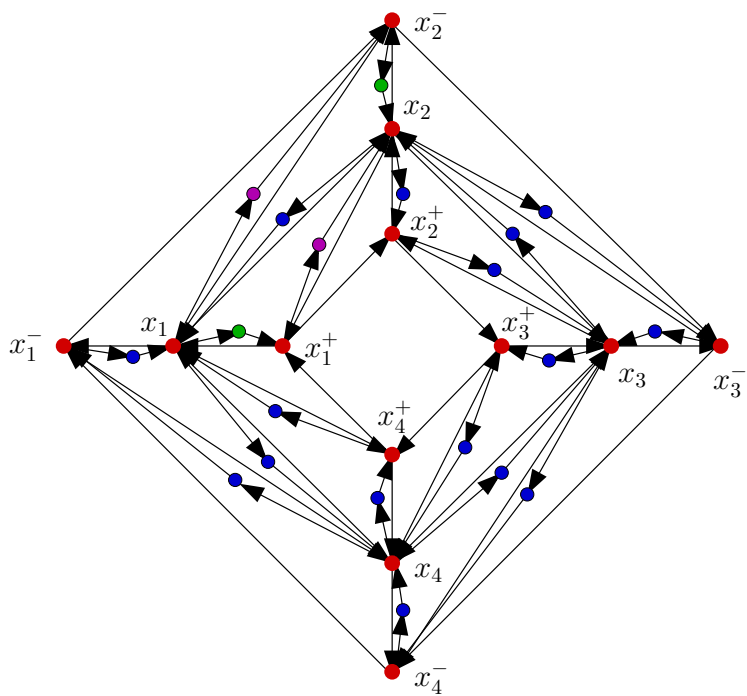}}
    \caption{Adjacent vertical closed cycles $M^-, M, M^+$ with frozen vertices added.  The cluster $R$-matrix $\tR_M$ swaps some of the frozen variables, such as $X_{1^+,1}$ with $X_{2,2^-}$ (shown in green), and $X_{2,1^+}$ with $X_{2^-,1}$ (shown in purple).}
    \label{fig:clR3}
\end{figure}
%

For $j \in \Z/n \Z$, we define a sequence of $2n-2$ mutations and $2n+1$ permutations by 
\begin{align}\label{eq:tR-op}
\begin{split}
\tcR_{M,j} &:= \prod_{i \in \Z/n\Z} s_{v_{i^+,i},v_{i+1,i+1^-}} \circ \prod_{i \in \Z/n\Z} s_{v_{i,i-1^+},v_{i^-,i-1}}
\circ s_{j-2,j-1} \\
& \qquad \circ (\mu_{j} \circ \mu_{j+1} \cdots \circ \mu_{j-4} \circ \mu_{j-3}) \circ 
(\mu_{j-1} \circ \mu_{j-2} \circ \cdots \circ \mu_{j+1} \circ \mu_j).
\end{split}
\end{align}

The operation $\tcR_{M,j}$ acts both on the quiver $\tQ_{n,m}$, sending it to $\tcR_{M,j}(\tQ_{n,m})$, and acts as a rational transformation of $\tF$. 

\begin{thm}\label{thm:Rcluster2}
For the enriched quiver $\tilde Q_{n,m}$, let $(\bx_{Q_{n,m}},\bX_{Q_{n,m}})$ be the cluster consisting of mutable variables and frozen variables.  For any $j \in \Z/n\Z$, we have 
$$\tcR_{M,j}(\tilde Q_{n,m}, \bx_{Q_{n,m}}, \bX_{Q_{n,m}}) 
= (\tilde Q_{n,m}, \tR_M(\bx_{Q_{n,m}}), \tR_M(\bX_{Q_{n,m}})).$$
\end{thm}

\begin{thm} \label{thm:Rbraid}
The $m-1$ transformations $\tR_{M_1},\tR_{M_2},\ldots,\tR_{M_{m-1}}\colon \tF \to \tF$ satisfy $\tR_{M_i} \tR_{M_{i+1}} \tR_{M_i} = \tR_{M_{i+1}} \tR_{M_i} \tR_{M_{i+1}}$ and $\tR_{M_i}^2 = \id$.  Furthermore, for $i,j$ such that $|i-j|>1$ we have  
$\tR_{M_i} \tR_{M_j} = \tR_{M_j} \tR_{M_i}$.
Thus $\tR_{M_i}$ generates an action of the symmetric group $\mS_m$ on $\tF$.
\end{thm}

See \S \ref{subsec:Rcluster2} and \S \ref{subsec:Rbraid} for the proofs.

%
%

\section{Geometric $R$-matrix}\label{sec:geom}

\subsection{Definition of the geometric $R$-matrix}\label{subsec:geomR}

Let $\p = (p_1,p_2,\ldots,p_n)$ and $\q = (q_1,q_2,\ldots,q_n)$, where indices are taken modulo $n$.
Define the polynomials $\kappa_i(\p,\q) \in \Z[\p,\q]$ for $i \in \Z/n\Z$ by
$$
\kappa_i(\p,\q) = \sum_{j=0}^{n-1} \prod_{\ell=1}^{j} p_{i-\ell} \prod_{\ell=j+2}^{n} q_{i-\ell}.
$$

Define the rational map $R\colon \QQ(\p,\q) \to \QQ(\p,\q)$, called the {\it geometric $R$-matrix} by the formula
$$
R(p_i) = q_i\frac{\kappa_{i+1}(\p,\q)}{\kappa_i(\p,\q)}, \qquad R(q_i) = p_i \frac{\kappa_{i}(\p,\q)}{\kappa_{i+1}(\p,\q)}.
$$

The geometric $R$-matrix appears in the theory of geometric crystals \cite{KNO,Et}, in study of total positivity in loop groups \cite{TP} and in other places \cite{KNY}. 

\begin{example}
 For $n=4$ we have
 $$R(q_1) = p_1\frac{q_1q_2q_3+p_4q_1q_2+p_3p_4q_1+p_2p_3p_4}{q_2q_3q_4+p_1q_2q_3+p_4p_1q_2+p_3p_4p_1}.$$
\end{example}

%
%

%

\subsection{From the cluster $R$-matrix to the geometric $R$-matrix}
\label{sec:xtoq}
We picture the $2n$ variables $\p$ and $\q$ as lying on a cylindrical network as indicated in Figure \ref{fig:clR4}.  This network consists of two parallel vertical wires going around the cylinder, and $n$ horizontal wires going from left to right.  The variables $\p$ and $\q$ are placed at the $2n$ vertices.  

\begin{figure}[ht]
\scalebox{0.7}{\input{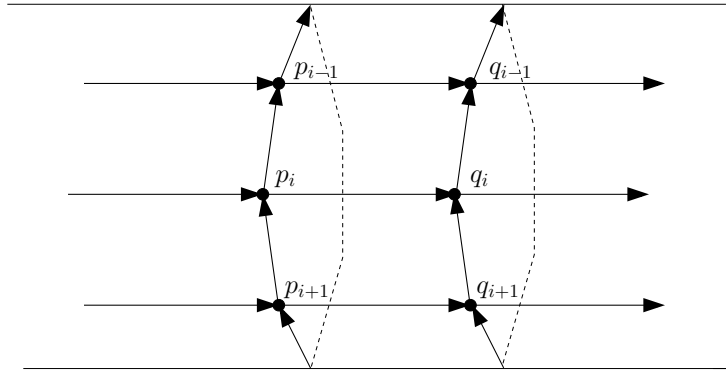}}
    \caption{Adjacent parallel vertical wires on a cylinder.}
    \label{fig:clR4}
\end{figure}

Associate a {\it {chamber variable}}, or {\it {$\tau$-function}} $x$ to each of the chambers into which the wires partition the cylinder.  Index these chamber variables so that the vertex $q_i$ is surrounded by $x_i, x_{i+1}, x_i^+, x_{i+1}^+$, while $p_i$ is surrounded by 
$x_i^-, x_{i+1}^-, x_i, x_{i+1}$; see Figure \ref{fig:clR5}. We may superimpose the triangular grid quiver $Q_{n,2}$ onto the network,
where we identify $M_0,M_1,M_2 \subset Q_{n,2}$ with $M^-,M,M^+$ respectively.
In the corresponding enriched quiver $\tilde Q_{n,2}$ remove all the frozen variables of the form $X_{i,i+1}$ and $X_{i,i^-}$ (or equivalently, set those frozen variables to $1$).  
Call the resulting quiver $\tilde Q'_{n,2}$.

The above procedure is a special case of a general way to associate a quiver to a wiring diagram. Specifically, for any chamber of the wiring diagram there are generically six arrows in the quiver incident to it. They alternate in direction and connect the chamber
to the three ``most forward'' and the three ``most backward'' chambers neighboring it, as shown in Figure \ref{fig:clR13}. In special cases, for example when the original chamber being a triangle for example, some 
of the six arrows cancel each other out and one is left with only four arrows. 
\begin{figure}[ht]
\scalebox{0.5}{\input{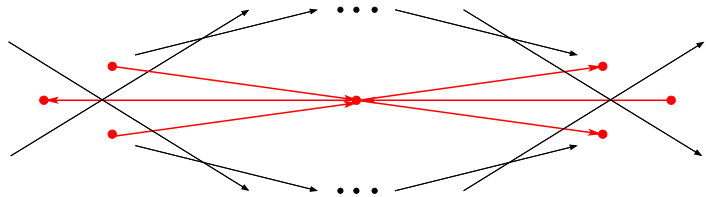}}
    \caption{Associating a quiver to a wiring diagram.}
    \label{fig:clR13}
\end{figure}
This construction is usually done for {\it {planar}} wiring diagrams, see \cite{FZ3} for an equivalent rule.  Our cylindric networks are locally planar, and 
we obtain $\tilde Q'_{n,m}$ applying the construction to a network $N_{n,m}$ 
defined after Theorem~\ref{thm:iota-RR}.

\begin{figure}[ht]
\scalebox{0.7}{\input{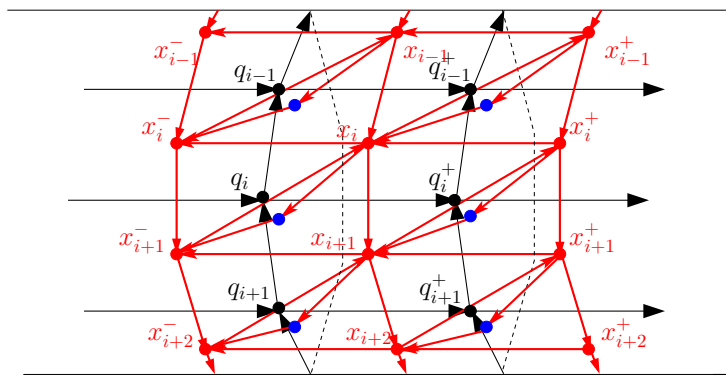}}
    \caption{The triangular grid quiver superimposed with the cylindrical network.}
    \label{fig:clR5}
\end{figure}

Define an algebra map $\iota\colon \QQ(\p,\q) \to \F(\tilde Q'_{n,2}) =\QQ(\x^-,\x,\x^+,X_{i+1^-,i},X_{i+1,i^+})$ by 
\begin{equation}\label{eq:x2q} \iota(p_i) = \frac{x_{i}^- x_{i+1} }{x_{i+1}^- x_i}X_{i+1^-,i},
 \qquad \iota(q_i) = \frac{x_i x_{i+1}^+ }{x_{i+1} x_i^+}X_{i+1,i^+}
\end{equation}
for all $i \in \Z/n\Z$.  It is clear that $\iota$ is an inclusion of fields.  Abusing notation, we still denote the cluster $R$-matrix (obtained by specializing certain frozen variables to $1$ in \eqref{eq:RonXx} and \eqref{eq:RonX}) of $\F(\tilde Q'_{n,2})$ by $\tR_{\x}$.

\begin{thm}\label{thm:iota-RR}
We have $\iota \circ R = \tR_{\x} \circ \iota$.
\end{thm}

\begin{proof}
We first compute that
\begin{align*}
\iota(\kappa_i(\p,\q))& = 
\sum_{j=0}^{n-1} 
\left(\prod_{\ell=1}^{j} \frac{x_{i-\ell}^- x_{i+1-\ell} }{x_{i+1-\ell}^- x_{i-\ell}}X_{i+1-\ell^-,i-\ell} \right) \left(\prod_{\ell=j+2}^{n} \frac{x_{i-\ell} x_{i+1-\ell}^+ }{x_{i+1-\ell} x_{i-\ell}^+}X_{i+1-\ell,i-\ell^+}\right)\\
&=\frac{x_i^2}{x_i^-x_i^+}\sum_{j=0}^{n-1} \frac{x^-_{i-j} \, x_{i-j-1}^+}{ x_{i-j}\, x_{i-j-1}} \prod_{\ell=1}^j X_{i+1-\ell^-,i-\ell} \prod_{\ell=j+2}^{n}X_{i+1-\ell,i-\ell^+}.
\end{align*}
Thus
\begin{align*}
&\iota(R(p_i)) \\
&=\iota(q_i)\frac{\iota(\kappa_{i+1}(\p,\q))}{\iota(\kappa_i(\p,\q))} \\
&=\frac{x_i^-\, x_{i+1}}{x_{i+1}^-\, x_i} \, X_{i+1,i^+} \, \frac{\displaystyle{\sum_{j=0}^{n-1} x^-_{i+1-j} \left(\prod_{r \in \Z/n\Z \setminus\{i-j,i+1-j\}} x_i \right) x_{i-j}^+ \prod_{\ell=1}^j X_{i+2-\ell^-,i+1-\ell} \prod_{\ell=j+2}^{n}X_{i+2-\ell,i+1-\ell^+}}}{\displaystyle{\sum_{j=0}^{n-1} x^-_{i-j} \left( \prod_{r \in \Z/n\Z \setminus\{i-1-j,i-j\}} x_i \right) x_{i-j-1}^+ \prod_{\ell=1}^j X_{i+1-\ell^-,i-\ell} \prod_{\ell=j+2}^{n}X_{i+1-\ell,i-\ell^+}}}.
\end{align*}
Re-indexing the summations, we see that this is equal to
\begin{align*}
\tR_\x(\iota(p_i)) &=\frac{x_i^-\, x_{i+1}}{x_{i+1}^-\, x_i} \, X_{i+1,i^+} \, \frac{\displaystyle{\sum_{j=1}^n  x_{j+1}^- \left(\prod_{\ell = j+2}^{j-1} x_{\ell} \right) x_j^+   \prod_{\ell=j+1}^{i}  X_{\ell+1^-,\ell}\prod_{\ell=i+2}^j  X_{\ell,\ell-1^+} }}{\displaystyle{\sum_{j=1}^n  x_{j+1}^- \left(\prod_{\ell = j+2}^{j-1} x_{\ell} \right) x_j^+   \prod_{\ell=j+1}^{i-1}  X_{\ell+1^-,\ell}\prod_{\ell=i+1}^j  X_{\ell,\ell-1^+} }},
\end{align*}
as required.
\end{proof}

Let $N_{n,m}$ denote a network on a cylinder with $n$ parallel horizontal 
wires and $m$ parallel vertical wires $W_1, W_2, \ldots,W_m$ 
which go around the cylinder (oriented as in Figure \ref{fig:clR4}).  Each circle $W_j$ has $n$ crossings with 
the horizontal wires. Let $q_{j,i}$ be a variable assigned to the 
$i$-th crossing on $W_j$, where we assume $i \in \Z/n\Z$.

Let $\Q = \Q(N_{n,m}) := \QQ(\q_1,\ldots,\q_m)$ where $\q_j = (q_{j,i})_{i=1,\ldots,n}$.  Define the quiver $\tilde Q'_{n,m}$ by removing frozen vertices as before.  Define an algebra map $\iota_m \colon \Q \to \F(\tilde Q'_{n,m})$ by
\begin{align}\label{eq:xX-q}
  \iota_m(q_{j,i}) 
  = 
  \frac{x_{j-1,i} \, x_{j,i+1}}{x_{j-1,i+1}\, x_{j,i}} X_{v_{(j-1,i+1),(j,i)}}.
\end{align}
This reminds us of the $\tau$-function substitution, where the cluster variables play the role of $\tau$-functions.
The following result is well known \cite{KNY,Et,TP}.  From Theorem \ref{thm:Rbraid}, we obtain a new proof via the cluster $R$-matrix.

\begin{cor}
Let $R_j$, for $j = 1,2,\ldots,m-1$, act on the variables $\q_j$ and $\q_{j+1}$ by $R$ and fixing all the other variables in $\Q$.  Then we have $R_j R_{j+1} R_j = R_{j+1} R_j R_{j+1}$ and $R_j^2 = \id$.  Furthermore, for $i,j$ such that $|i-j|>1$ we have $R_i R_j = R_j R_i$.  Thus $R_j$ generates an action of $\mS_m$ on $\Q$.
\end{cor}

\section{Quantum geometric $R$-matrix} \label{sec:qgeom}
\subsection{Quantum torus}\label{ssec:qtorus}
We refer the reader to \cite[Section 4]{BZ} for a more systematic treatment of quantum tori.

Let $\p = (p_1,p_2,\ldots,p_n)$ and $\q = (q_1,q_2,\ldots,q_n)$ where the indices are taken modulo $n$.
Let $\Z_\e[\p,\q]$ be the noncommutative $\Z[\e,\e^{-1}]$-algebra generated by $p_i,p_i^{-1},q_i,q_i^{-1}$ with the commutation relations
$$
p_iq_j = \begin{cases} \e q_j p_i & \mbox{ if $j= i$ or $j = i-2$,} \\ 
\e^{-2} q_{j} p_i  & \mbox{ if $j = i-1$,} \\
q_j p_i & \mbox{otherwise,}
\end{cases}
$$
and for $\r = \p$ or $\r = \q$, the relations
$$
r_i r_j = \begin{cases} \e r_j r_i & \mbox{if $j = i-1$,} \\
\e^{-1} r_j r_i & \mbox{if $j = i+1$,} \\
r_j r_i & \mbox{otherwise.}
\end{cases}
$$
Note that the relations are preserved by a $\Z/n\Z$-shift of the indices.
The subalgebra $\Z[\e,\e^{-1}]$ is central.  
The ring $\Z_\e[\p,\q]$ is a {\it quantum torus}; it has a distinguished $\Z[\e,\e^{-1}]$-basis given by the monomials $p_1^{a_1}p_2^{a_2} \cdots p_n^{a_n} q_1^{b_1}q_2^{b_2} \cdots q_n^{b_n}$, where $a_i,b_i \in \Z$.  
The quantum torus $\Z_\e[\p,\q]$ is an Ore domain \cite[Appendix]{BZ}. 
It includes into its skew field of fractions $\QQ_\e \langle \p,\q \rangle$.

\subsection{Definition of the quantum geometric $R$-matrix}
Define the {\it quantum geometric $R$-matrix} $R^\e\colon \QQ_\e\langle \p,\q \rangle \to \QQ_\e\langle \p,\q \rangle$ by
\begin{align}\label{eq:R-pq}
  R^\e (p_i) = (\kappa_{i}^\e(\p,\q))^{-1} \cdot q_{i} \cdot \kappa_{i+1}^\e(\p,\q), \qquad R^\e(q_i) = (\kappa_{i+1}^\e(\p,\q))^{-1} \cdot p_i \cdot \kappa_{i}^\e(\p,\q)
\end{align} 
where
\begin{align}\label{eq:q-kappa}
  \kappa_i^\e(\p,\q) 
  := 
  \sum_{j=0}^{n-1} p_{i-1}p_{i-2}\cdots p_{i-j}\, q_{i-j-2} q_{i-j-3} \cdots q_{i-n} \in \QQ_\e \langle \p, \q \rangle,
\end{align}
for $i \in \Z/n\Z$.
Though the polynomials $\kappa_i^\e$ are apparently the same as the classical case, the variables $p_i$ and $q_i$ are now $\epsilon$-commuting.  It follows from Theorem \ref{thm:commpres} below that $R^\e$ is a morphism of 
skew fields.

\subsection{Yang-Baxter relation}\label{ssec:Qe}

Recall from \S\ref{sec:xtoq} the definition of the cylindrical network $N_{n,m}$.
Let $S_k$ be the $k$-th {\it snake path} in $N_{n,m}$, consisting of the variables $q_{j,i}$ satisfying $i+j \equiv 1+k \mod n$.
An example of a snake path is shown in Figure \ref{fig:clR14}.

\begin{figure}[ht]
\scalebox{0.5}{\input{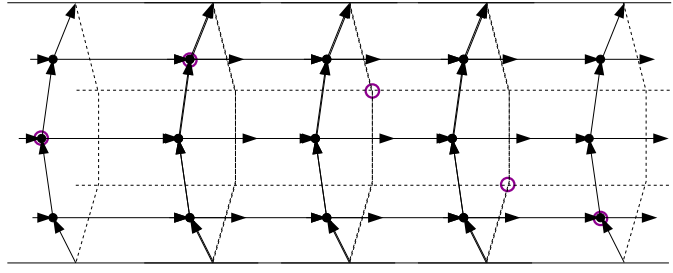}}
    \caption{Circled vertices form a snake path in $N_{5,m}$.}
    \label{fig:clR14}
\end{figure}

\begin{definition}\label{def:Q}
Let $\mathcal{Q}_\e = \QQ_\e\langle \q_1,\q_2,\ldots,\q_m \rangle$ be the skew field of the quantum torus over $\Z[\e,\e^{-1}]$ generated by 
$q_{j,i}~(1 \leq j \leq m, ~i \in \Z/n\Z)$ with the commutation relations:
\begin{align*}
  q_{ji} q_{j'i'} = 
  \begin{cases}
  \e \, q_{j'i'} q_{ji} & q_{ji} \in S_{k+1}, q_{j'i'} \in S_{k}, j \leq j' 
  \\
  \e^{-1} q_{j'i'} q_{ji} & q_{ji} \in S_{k+1}, q_{j'i'} \in S_{k}, j > j'
  \\ 
  \e^{-2} q_{j'i'} q_{ji} & q_{ji}, q_{j'i'} \in S_{k}, j < j'
  \\
  \e^{2} q_{j'i'} q_{ji} & q_{ji}, q_{j'i'} \in S_{k}, j > j' 
  \\
  q_{j'i'} q_{ji} & \text{otherwise}.
  \end{cases}
\end{align*}
\end{definition}
Note that the commutation relations restricted to the variables $\q_j$ and $\q_{j+1}$ are the same as those for $\p$ and $\q$ in $\QQ_\e\langle \p,\q\rangle$.  These relations were in part inspired by \cite{Ber}.

Let $R^\e_j$, for $j = 1,2,\ldots,m-1$, act on the variables $\q_j$ and $\q_{j+1}$ by $R^\e$ and fixing all the other variables in $\Q_\e$.

\begin{thm} \label{thm:commpres}
The map $R^\e_j\colon \Q_\e \to \Q_\e$ is a morphism of skew fields.
Thus if $R^\e(\q_j,\q_{j+1}) = (\q'_j,\q'_{j+1})$ then $(\q_1,\ldots,\q'_j,\q'_{j+1},\ldots,\q_m)$ satisfy identical commutation relations to $(\q_1,\ldots,\q_m)$.  
\end{thm}

\begin{thm}\label{thm:q-qYBR}
We have $R^\e_j R^\e_{j+1} R^\e_j = R^\e_{j+1} R^\e_j R^\e_{j+1}$ and $(R_j^\e)^2 = \id$.  Furthermore, for $i,j$ such that $|i-j|>1$ we have  $R_i^\e R_j^\e = R_j^\e R_i^\e$.  Thus $R^\e_j$ generates an action of $\mS_m$ on $\Q_\e$.
\end{thm}

See \S\ref{subsec:commpres} and \S\ref{subsec:YBRqR-yR} for the proofs.

\section{Invariants of the quantum geometric $R$-matrix} \label{sec:inv}

We study the invariants
of the quantum geometric $R$-matrix action on the network $N_{n,m}$. 
For this purpose, we use the notion of {\it highway paths} on $N_{n,m}$,
as studied in \cite{LP}, and we define quantum analogues of the loop symmetric functions of \cite{TP}.

In this section only, we use the notation $q_j^{(i)}:= q_{j,i}$ to more closely match the usual notation in the theory of (loop) symmetric functions.
 
\subsection{Highway Measurements}\label{ssec:highway}

Consider the class of networks which are 
\begin{itemize}
 \item directed graphs embedded on an orientable surface, where edges intersect only at vertices;
 \item have boundary vertices of degree $1$, either sources or sinks depending on the direction of the adjacent edge;
 \item have internal vertices of degree $4$, with exactly $2$ incoming edges and $2$ outgoing edges, and incoming and outgoing edges not interlacing;
 \item have a weight assigned to each vertex, and the weights do not necessarily commute. 
\end{itemize}

The cylindric network $N_{n,m}$ (see \S\ref{sec:geom}) with weights $q_{j}^{(i)}$ is an example of such a network. 
Given such a network $N$, we consider {\it {highway measurements}} $M(N)$ in these networks defined as follows.  We fix ordered collections of source boundary vertices $s_1,s_2,\ldots,s_r$ and sink boundary vertices $t_1,t_2,\ldots,t_r$ of equal cardinality, and for each $i = 1,2,\ldots,r$, we fix a homology class $[\gamma_i]$ of a path from $s_i$ to $t_i$.  We then take the weight generating function
$$
M(N) = M_{s,t,[\gamma]}(N) := \sum_{\P} \wt(\P)
$$
of families of paths $\P = (P_1,P_2,\ldots,P_r)$, where $P_i$ goes from $s_i$ to $t_i$, has homology class $[P_i] = [\gamma_i]$, and the paths
are required to be {\it non-intersecting}, that is, they may not share common edges.  
\begin{figure}[ht]
\scalebox{0.8}{\input{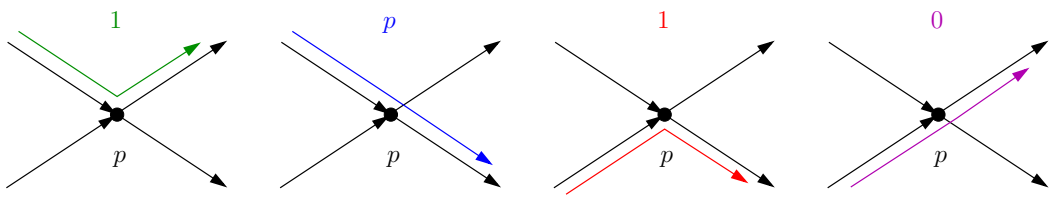}}
    \caption{The four ways a path can pass through a vertex of a network, and the corresponding weight it picks up.}
    \label{fig:clR6}
\end{figure}
The weight assigned to each path $P$ is the product of the weights picked up at each of its vertices, depending on one of the four ways the path passes through the vertex, as shown in Figure \ref{fig:clR6}. 
For each path $P$ we take the product in the order the vertices are visited to obtain the weight $\wt(P)$.  The weight of a family $\P = (P_1,P_2,\ldots,P_r)$ of paths is the product $\wt(\P) = \wt(P_1) \wt(P_2) \cdots \wt(P_r)$.  Note that the ordering of the paths is important since we allow weights to be noncommutative.  We call a path $P$ a {\it highway path} if it has nonzero weight.

In \S\ref{ssec:elem}, we consider path families on $N_{n,m}$ consisting of a single path.  In \S\ref{ssec:loopSchur}, we consider path families on the universal cover $\tilde N_{n,m}$.  In \S\ref{subsec:q-cylSchur}, we consider arbitrary path families on $N_{n,m}$.  In the latter cases, we will impose conditions on the ordering of paths in path families.

%

\subsection{Quantum loop elementary symmetric functions}\label{ssec:elem}

Let $\Q_\e$ be the skew field defined in Definition \ref{def:Q}, with distinguished noncommuting generators now denoted $q_j^{(i)}$.  
Denote by $\Q_\e^{\mS_m}$ the invariants of the quantum $R$-matrix action on $\Q_\e$ from Theorem \ref{thm:q-qYBR}.

For $k = 1,2,\ldots,m$ and $r \in \Z/n\Z$, define the quantum loop elementary symmetric function 
$$
e_k^{(r)}(\q_1,\ldots,\q_m):= \sum_{1 \leq j_1<j_2<\cdots<j_k\leq m} q^{(r+1-j_1)}_{j_1} q^{(r+2-j_2)}_{j_2} \cdots q^{(r+k-j_k)}_{j_k} \in \Q_\e.
$$
For example, with $n = 3$ and $m = 4$, we have
\begin{align*}
e_1^{(1)} &= q_1^{(1)} + q_2^{(3)} + q_3^{(2)} + q_4^{(1)} \\
e_2^{(1)} &= q_1^{(1)}q_2^{(1)} + q_1^{(1)}q_3^{(3)}+ q_1^{(1)}q_4^{(2)} + q_2^{(3)}q_3^{(3)} + q_2^{(3)}q_4^{(2)} + q_3^{(2)}q_4^{(2)} \\ 
e_3^{(1)} &= q_1^{(1)} q_2^{(1)} q_3^{(1)} + q_1^{(1)} q_2^{(1)} q_4^{(3)} +q_1^{(1)} q_3^{(3)} q_4^{(3)} +q_2^{(3)} q_3^{(3)} q_4^{(3)} \\
e_4^{(1)} &=q_1^{(1)} q_2^{(1)} q_3^{(1)} q_4^{(1)}.
\end{align*}
By convention, we have $e_0^{(r)} = 1$ and $e_{s}^{(r)} = 0$ for $s < 0 $ or $s > m$.  The quantum loop elementary symmetric function $e_k^{(r)}$ can be interpreted as a highway measurement: it is the weight generating function of paths in $N_{n,m}$ with a fixed sink and fixed source.  We shall show that $e_k^{(r)} \in \Q_\e^{\mS_m}$.

For two matrices $A = (a_{ij})$ and $B = (b_{ij})$ whose entries lie in a possibly noncommutative ring, we define the product $AB$ by $(AB)_{ij} = \sum_k a_{ij} b_{jk}$, where the ordering of the factors in the product $a_{ij} b_{jk}$ must be kept. 

Let 
$$
M(\bq;t) = \begin{pmatrix} q_1 & 0 & 0 & t \\
1 & q_2 & 0 & 0 \\
0 & \ddots & \ddots &0 \\
0 & 0 & 1 & q_n
\end{pmatrix},
$$
where $t$ is a spectral parameter.  Let $\M(t):= M(\q_1;t)M(\q_2;t) \cdots M(\q_m;t)$.  The following observation follows from the definitions.
\begin{lem}\label{lem:e}
We have
$$
\M_{i,j}(t) = \sum_{s \geq 0} e^{(i)}_{j-i+m-sn} \, t^s.
$$
\end{lem}

Write $(\p',\q') := R^\e(\p,\q)$.  

\begin{lem}\label{lem:RpqM}
We have $M(\p;t)M(\q;t) = M(\p';t) M(\q';t)$.
\end{lem}
\begin{proof}
We check that $p_i q_i = p_i' q_i'$ and 
$p_{i+1} + q_i = p_{i+1}' + q_i'$ follow from \eqref{eq:R-pq}.	
\end{proof}

\begin{cor} \label{cor:elem}
For $k >0$ and $r \in \Z/m\Z$, we have $e_k^{(r)}(\q_1,\ldots,\q_m) \in \Q_\e^{\mS_m}$.
\end{cor}
\begin{proof}
By Lemma \ref{lem:RpqM}, the entries of $\M(t)$ belong to $\Q_\e^{\mS_m}$.  The result follows from Lemma \ref{lem:e}.
\end{proof}

We remark that the classical ($\e =1$) case of $e_k^{(r)}$ was essentially   
studied in \cite[Section 2]{Yam01}.

\subsection{Yang-Baxter move}
We now consider certain {\it {local moves}} that do not change the highway measurements. The most important one is the {\it {Yang-Baxter move}} shown in Figure \ref{fig:clR7}. 
\begin{figure}[ht]
\scalebox{0.8}{\input{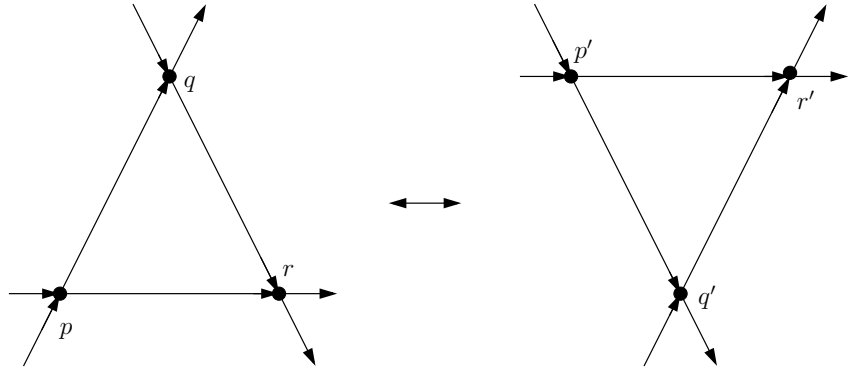}}
    \caption{The Yang-Baxter move.}
    \label{fig:clR7}
\end{figure}

Assume $p,q,r$ are noncommutative parameters which nevertheless satisfy the relation $$pqr=rqp.$$  Define the transformation of parameters in the Yang-Baxter move as follows:\begin{equation}\label{eq:YB}
\begin{cases}
p' = qr(p+r)^{-1} = (p+r)^{-1}rq,\\
q' = p+r,\\
r' = qp(p+r)^{-1} = (p+r)^{-1}pq.
\end{cases}
\end{equation}
In the classical $\e = 1$ case, this Yang-Baxter move was studied in \cite{LP}.  In our setting it can be compared to \cite{Ber}.

\begin{prop} \label{prop:YB} \
\begin{enumerate}
\item 
The Yang-Baxter transformation preserves highway measurements locally, that is, all highway measurements of the two networks shown in Figure \ref{fig:clR7} are identical.
\item
The resulting parameters again satisfy $p'q'r'=r'q'p'$.
\item
The Yang-Baxter transformation is an involution. 
\end{enumerate}
\end{prop}

\begin{proof}
 The only non-trivial part to verify is that the highway measurements are preserved locally. The four essentially different ways to take highway measurements are shown in Figure \ref{fig:clR8}. 
 \begin{figure}[ht]
\scalebox{0.5}{\input{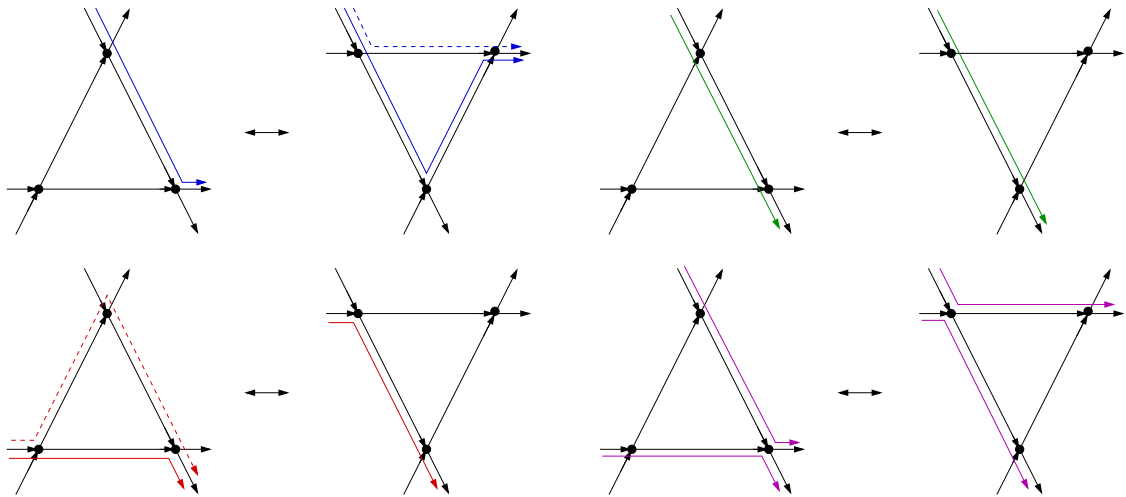}}
    \caption{Non-trivial highway measurements in the Yang-Baxter move.}
    \label{fig:clR8}
\end{figure}
This results in the following system of equations:
$$
\begin{cases}
qr=p'q',\\
p+r=q',\\
q=p'+r',\\
qp=r'q',\\
pq=q'r'.
\end{cases}
$$
Note that we include both equations $qp=r'q'$ and $pq=q'r'$ since we want the measurements to be conserved no matter which of the two orders we choose on the two distinct paths we have in this case. It is easy to check now that the transformation above satisfies all those relations.
\end{proof}

The other local move we want is the {\it {lens creation-annihilation}} move. 
 \begin{figure}[ht]
\scalebox{0.7}{\input{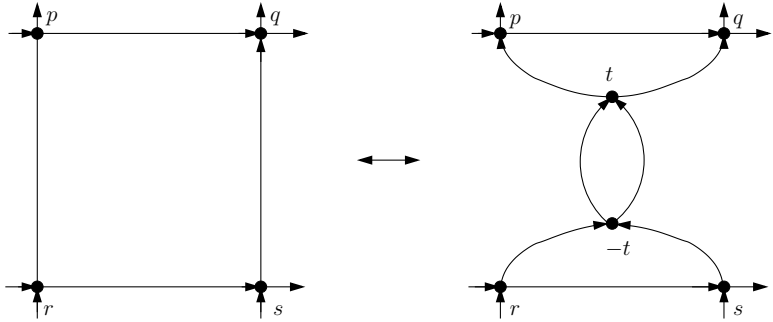}}
    \caption{The lens creation-annihilation move.}
    \label{fig:clR9}
\end{figure}

\begin{prop} \label{prop:lca}
The lens creation-annihilation transformation preserves highway measurements. 
\end{prop}

\begin{proof}
 The proof is the same as for the the Yang-Baxter move. In fact, in this case the measurements with more than one path do not impose any extra relations, and thus not only local but also global measurements are preserved.  In other words, if we perform the lens creation-annihilation transformation locally in a big network $N$, all highway measurements of $N$ are preserved.
\end{proof}
\subsection{Quantum loop Schur functions}\label{ssec:loopSchur}

Let $\lambda/\mu$ be a skew shape.  Let $T$ be a semistandard Young tableau of shape $\lambda/\mu$ and $r \in \Z/n\Z$.  Define the content $c(b)$ of a box $b$ in the $i$-th row and $j$-th column to be $c(b)=i-j$.
The {\it $r$-reading word} of $T$ is the monomial $\wte^{(r)}(T) = \prod_{b \in \lambda/\mu} q_{T(b)}^{(c(b)-T(b)+r+1)} \in \Q_\e$, where the product is taken by reading the columns from the left to the right, and each column is read from the top to the bottom.  For example, with $n = 3$,
$$
\text{ for } T = \tableau[sY]{1 & 1 & 2 &4 \\ 2 & 3 &3 \\ 4} \qquad \text{we have} \qquad \wte^{(1)} = (q^{(1)}_1 q^{(1)}_2 q^{(3)}_4)(q^{(3)}_1 q^{(2)}_3) (q^{(1)}_2 q^{(1)}_3) (q^{(1)}_4).
$$

Define the {\it quantum loop Schur function} by 
$$
s^{(r)}_{\lambda/\mu}(\q_1,\ldots,\q_m) = \sum_T \wte^{(r)}(T) \in \Q_\e
$$
where the summation is over all semistandard Young tableaux $T$ of shape $\lambda/\mu$ filled with the integers $1,2,\ldots,m$.  Up to a shift in the upper index, our quantum loop Schur functions reduce to the loop Schur functions of \cite{TP} at $\e = 1$.  When $\lambda/\mu$ is a column, the quantum loop Schur function is simply the quantum loop elementary symmetric function.

\begin{thm} \label{thm:schur}
For any skew shape $\lambda/\mu$ and $r \in \Z/n\Z$, 
we have $s_{\lambda/\mu}(\q_1,\ldots,\q_m) \in \Q_\e^{\mS_m}$.
\end{thm}

We remark that $s_{\lambda/\mu}(\q_1,\ldots,\q_m)$ is not a highway measurement of $N_{n,m}$ but of the universal cover $\tilde N_{n,m}$.  This is because it corresponds to families of paths that do not intersect on $\tilde N_{n,m}$, but whose images in $N_{n,m}$ {\it may} intersect.

\begin{example}\label{ex:loopschur}
Set $n=m=3$, $\lambda = (2,1)$ and $\mu = \emptyset$.
The eight semistandard tableaux of shape $\lambda$ are 
\begin{equation}\label{eq:21tableaux}
\tableau[sY]{1 & 1 \\ 2} \quad \tableau[sY]{1 & 2 \\ 2}
\quad \tableau[sY]{1 & 3 \\ 2}
\quad \tableau[sY]{1 & 1 \\ 3} \quad \tableau[sY]{1 & 2 \\ 3}
\quad \tableau[sY]{1 & 3 \\ 3} \quad \tableau[sY]{2 & 2 \\ 3}
\quad \tableau[sY]{2 & 3 \\ 3} \, .
\end{equation}
Thus we have 
$$
s_\lambda^{(1)}(\q_1,\q_2,\q_3) =
q_1^{(1)} q_2^{(1)} (q_1^{(3)} + q_2^{(2)} + q_3^{(1)})
+ q_1^{(1)} q_3^{(3)} (q_1^{(3)} + q_2^{(2)}+q_3^{(1)})
+ q_2^{(3)} q_3^{(3)} (q_2^{(2)}+q_3^{(1)}),
$$
where each weight in $s_\lambda^{(1)}(\q_1,\q_2,\q_3)$ 
corresponds to a pair of paths in $\tilde N_{3,3}$. 
Note that the pair of paths with weight
$q_1^{(1)} q_2^{(1)} q_3^{(1)}$ intersect in $N_{3,3}$. 
\end{example}

This section will be devoted to the proof of Theorem~\ref{thm:schur}. We follow closely the ideas of \cite{LP}, but in the noncommutative setting.  Our treatment is briefer than the treatment in \cite[Section 6]{LP}.
 \begin{figure}[ht]
\scalebox{0.5}{\input{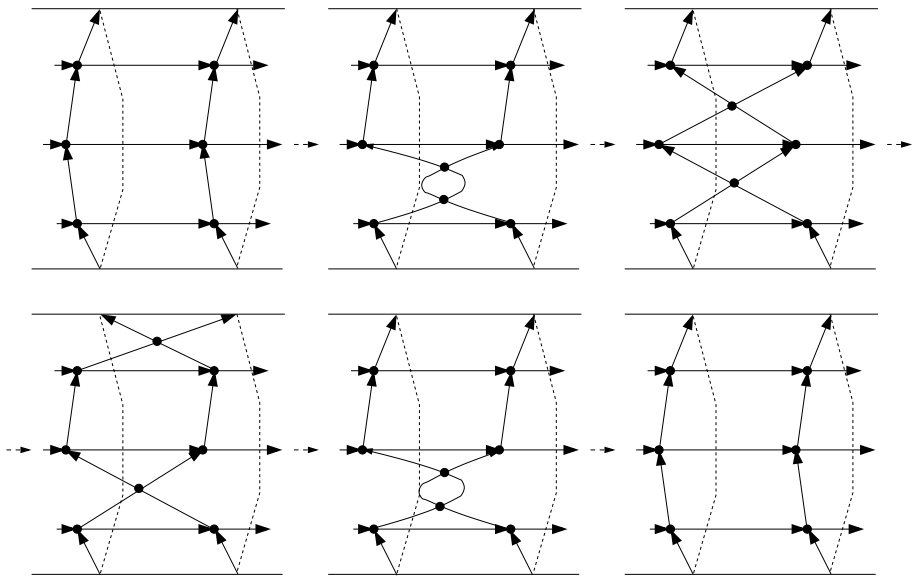}}
    \caption{Realization of quantum geometric $R$-matrix via local moves.}
    \label{fig:clR10}
\end{figure}
The idea is to realize the quantum geometric $R$-matrix by a sequence of local moves of $N_{n,m}$, each of which preserves highway measurements. We pick two adjacent vertical wires $W_j$ and $W_{j+1}$ and create a lens between them.  We push one of the two vertices of the lens around the cylinder through the horizontal wires using the Yang-Baxter move, and then annihilate the lens.
We choose the vertex parameter of the original lens so that it comes out on the other side equal to its original value, thus allowing the annihilation to happen.  The sequence of local moves is illustrated in Figure \ref{fig:clR10}.

Recall the definition of $\kappa_i^\e(\p,\q)$ \eqref{eq:q-kappa}.
Assume that the lens crossing being pushed is currently in the region with parameters $p_i, p_{i+1}, q_i, q_{i+1}$, and its value is equal to 
\begin{equation}\label{eq:r}
r_{i+1} = \left(\prod_{i=n}^1 p_i - \prod_{i=n}^1 q_i \right)(\kappa_{i+1}^{\epsilon})^{-1} = (\kappa_{i+1}^{\epsilon})^{-1} \left(\prod_{i=n}^1 p_i - \prod_{i=n}^1 q_i \right).
\end{equation}
The last equality holds since $\prod_{i=n}^1 p_i$ and $\prod_{i=n}^1 q_i$ are central elements of $\QQ_\e \langle \p, \q \rangle$. 

 \begin{figure}[ht]
\scalebox{0.7}{\input{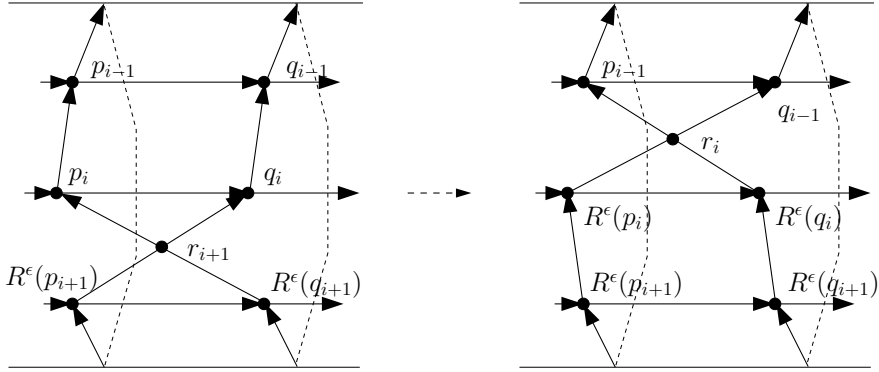}}
    \caption{A local move.}
    \label{fig:clR11}
\end{figure}

Let $\tilde N_{n,m}$ denote the universal cover of $N_{n,m}$.  It is a network embedded into a vertical strip.  We only consider highway measurements of $\tilde N_{n,m}$ where in a non-intersecting family $\P = (P_1,P_2,\ldots,P_r)$, the paths $P_i$ are ordered from bottom to top.

\begin{prop} \label{prop:lens}
With the above notation,
 \begin{enumerate}
  \item the parameters satisfy the relation $r_{i+1} p_i q_i = q_i p_i r_{i+1}$;
  \item applying the Yang-Baxter move to the parameters $r_{i+1}, p_i, q_i$, one obtains the parameters $R^{\epsilon}(p_i), R^{\epsilon}(q_i), r_i$;
  \item the Yang-Baxter move preserves all highway measurements on the universal cover $\tilde N_{n,m}$.
 \end{enumerate}
\end{prop}

\begin{proof}
 First, $p_i q_i = \epsilon q_i p_i$. Next, $(\prod_{j=i-1}^{i+1} q_i) q_i =  q_i (\prod_{j=i-1}^{i+1} q_i)$ and $(\prod_{j=i-1}^{i+1} q_i) p_i = \epsilon^{-1} p_i (\prod_{j=i-1}^{i+1} q_i)$. All other terms in $\kappa_{i+1}$ satisfy the same commutation relations 
 with $p_i$ and $q_i$, giving $$(\kappa_{i+1}^{\epsilon} - \prod_{j=i-1}^{i+1} q_i) q_i =  \epsilon^{-2} q_i (\kappa_{i+1}^{\epsilon} - \prod_{j=i-1}^{i+1} q_i) \;\;\; \text{and} \;\;\;
 (\kappa_{i+1}^{\epsilon} - \prod_{j=i-1}^{i+1} q_i) p_i = \epsilon p_i (\kappa_{i+1}^{\epsilon} - \prod_{j=i-1}^{i+1} q_i).$$
 Combining, we see that 
\begin{align}\label{eq:pq-kappa}
p_i q_i (\kappa_{i+1}^{\epsilon})  = \epsilon^{-1} (\kappa_{i+1}^{\epsilon}) p_i q_i.
\end{align}
Using $p_i q_i = \epsilon q_i p_i$, we obtain (1). 
 
 For the second claim, we apply the formulae \eqref{eq:YB} of the Yang-Baxter move. We have 
 $$(r_{i+1}+q_i)^{-1} r_{i+1} p_i  = \left[ (\kappa_{i+1}^{\epsilon})^{-1} \left( \prod_j p_j - \prod_j q_j +  \kappa_{i+1}^{\epsilon} q_i \right) \right]^{-1} (\kappa_{i+1}^{\epsilon})^{-1} \left( \prod_j p_j - \prod_j q_j\right) p_i   $$
 $$=  \left[  \prod_j p_j - \prod_j q_j +  \kappa_{i+1}^{\epsilon} q_i \right]^{-1} p_i \left( \prod_j p_j - \prod_j q_j\right)  =  \left[  p_i \kappa_{i}^{\epsilon} \right]^{-1} p_i \left( \prod_j p_j - \prod_j q_j\right)  = r_i.$$
 Also, in the process we saw that 
 $$r_{i+1}+q_i = (\kappa_{i+1}^{\epsilon})^{-1}  p_i \kappa_{i}^{\epsilon} = R^{\epsilon}(q_i),$$ and the argument for $R^{\epsilon}(p_i)$ is similar.  This proves (2).
 
 Finally, we argue that these Yang-Baxter moves preserve all highway measurements of $N_{n,m}$. The only non-trivial case is the case when there are two paths running through the parameters of the Yang-Baxter move, both picking up a non-trivial weight. 
 The reason this case is non-trivial is as follows: while we know from Proposition \ref{prop:YB} that highway measurements are preserved locally, in the highway measurement of $N_{n,m}$ those two parameters are separated by several other (noncommuting) factors, since they belong to different paths. 
 We call these factors {\it {the middle factors}}.
 This case is shown in Figure \ref{fig:clR12}. 
    \begin{figure}[ht]
\scalebox{0.7}{\input{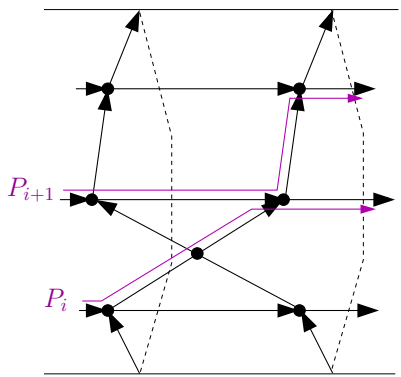}}
    \caption{The non-trivial case to consider for preservation of global measurements. }
    \label{fig:clR12}
\end{figure}

Since we assumed that paths in a non-intersecting family are ordered from bottom to top, the two highway paths involved in Figure \ref{fig:clR12} must be $P_i$ and $P_{i+1}$ for some $i$.
In this case the weight of the family of highway paths has the following form:
$$
\wt_\e(P_1) \wt_\e(P_2) \cdots \wt_\e(P_{i-1}) \dotsc r_{i+1}\, ( \text{ middle factors } )\, p_i \dotsc \wt_\e(P_{i+2}) \cdots \wt_\e(P_{r}).
$$ 
By Proposition \ref{prop:YB}(1), we have $r_{i+1} p_i = R^{\epsilon}(q_i) r_i$ and $p_i r_{i+1}  = r_i R^{\epsilon}(q_i)$. We still need to commute $r_{i+1}$ or $p_i$ to through the middle factors in order to be able to apply these
 relations, and then we need to commute the result back. However, we claim the following: $r_{i+1}$ and $R^{\epsilon}(q_i)$ satisfy the same commutation relations with the middle factors as $p_{i+1}$ and $q_i$ do. Indeed, it is clear from Figure \ref{fig:clR12} that all the middle factors come from vertices that are 
 strictly to the left or strictly to the right of the $p$-s and the $q$-s. This means the commutation relations are determined by the snake paths on which the variables lie. To see then that $(\kappa_{i+1}^{\epsilon})^{-1}$ satisfies the same relation as $p_{i+1}$ and $q_i$
 is a matter of a simple verification, which can be shortened by remembering that $\prod_i p_i$ and $\prod_i q_i$ are central, and that terms in $\kappa_{i+1}^{\epsilon}$ differ from these products by omitting either $p_{i+1}$ or $q_i$. 
\end{proof}

\begin{cor}\label{lem:qR}
The sequence of local moves in Figure \ref{fig:clR10}, with lens vertex weight given by \eqref{eq:r}, acts on the vertex weights of $N_{n,m}$ by the quantum geometric $R$-matrix.
\end{cor}
\begin{proof}
Since $r_{i+n} = r_i$ in \eqref{eq:r}, when the parameter is pushed through all $n$ wires, it acquires the original value, and thus one can carry out the lens annihilation move. Thus by Proposition \ref{prop:lens}(2), the sequence local moves in Figure \ref{fig:clR10} indeed realizes the quantum geometric $R$-matrix.
\end{proof}


\begin{proof}[Proof of Theorem \ref{thm:schur}]
It follows from the definition of $\wte^{(r)}(T)$ that quantum loop Schur functions are highway measurements of the universal cover $\tilde N_{n,m}$ of the network $N_{n,m}$.  Namely, we consider families of non-intersecting paths on $\tilde N_{n,m}$ with fixed sources and sinks, with the paths ordered from bottom to top.  Corollary \ref{lem:qR} applies  also to $\tilde N_{n,m}$.  By Proposition \ref{prop:lens}(3), it follows that the quantum loop Schur function is preserved by the quantum geometric $R$-matrix.
\end{proof}

\subsection{Quantum cylindric loop Schur functions}
\label{subsec:q-cylSchur}
To avoid certain degenerate situations, in this section we assume that $n > 2$.  Cylindric loop Schur functions were defined in \cite{LPS} where they were shown to be invariants of the geometric $R$-matrix.  We define their quantum analogues.

Fix an integer $s$.  Let $\Z^2$ denote the integer lattice with points $(x,y)$.  Define the cylinder $\CC_s$ to be the following quotient of integer lattice: 
$$\CC_s = \Z^2/(n-s,s)\Z.$$ 
In other words, $\CC_s$ is the quotient of $\Z^2$ by the shift that sends $(x,y)$ into $(x+n-s,y+s)$. The set $\CC_s$ inherits a natural partial order from that on $\Z^2$ given by the transitive closure of the cover relations $(x,y)<(x+1,b)$ and $(x,y)<(x,y-1)$.  A box in the $i$-th row and $j$-th column of a Young diagram has coordinates $(j,-i)$.

A {\it {cylindric skew shape}} $D$ is a finite convex subposet of $\CC_s$.  A semistandard Young tableau of a cylindric skew shape $D$ is a map $T\colon D \to \Z_{>0}$ satisfying $T(x,y)\leq T(x+1,y)$ and $T(x,y)<T(x,y-1)$ whenever the corresponding boxes lie in $D$. 

Let $D$ be a cylindric skew shape and $r \in \Z/n\Z$.  Viewing $D$ as an infinite periodic skew shape in $\Z^2$, fix a fundamental domain consisting of $n-s$ consecutive columns of $D$, and order these columns from left to right.
For a cylindric semistandard Young tableau $T$ of shape $D$, the {\it $r$-reading word} of $T$ is the monomial $\wte^{(r)}(T) = \prod_{b \in D} q_{T(b)}^{(c(b)-T(b)+r+1)} \in \Q_\e$, where the product is taken by reading the columns in the chosen order, and each column is read from the top to the bottom.  

\begin{example}
Set $n=4$, $s=2$, and consider the shape $D$ in $\CC_2 = \Z^2/(2,2)\Z$, whose restriction to two adjacent columns is given by the shape $\lambda = (2,1,1)$. 
For example, a cylindric semistandard Young tableau $T$ of shape $D$ is 
given by extending
$$
T = \tableau[sY]{1 & 4 \\ 2  \\ 4}
$$
periodically,  and we have 
$\wte^{(1)}(T) = (q^{(1)}_1 q^{(1)}_2q^{(4)}_4)(q^{(1)}_4).$
Note that some semistandard Young tableaux of shape $\lambda$ may not give cylindric semistandard Young tableaux, for example,
$$
\tableau[sY]{1 & 4  \\ 2 \\ 3}.
$$ 
\end{example}


For a cylindric skew shape $D$, we define the {\it {quantum cylindric loop Schur function}} by
\begin{align*}
s^{(r)}_{D}(\q_1,\ldots,\q_m) = \sum_{T} x^{\wte^{(r)}(T)} \in \Q_\e,
\end{align*}
where the summation is over all semistandard Young tableaux of cylindric skew shape $D$ filled with the integers $1,2,\ldots,m$.  Up to a shift in the upper index, our quantum cylindric loop Schur functions reduce to the cylindric loop Schur functions of \cite{LPS} at $\e = 1$.

\begin{thm}\label{thm:q-cylindric-loop-schur}
For any cylindric skew shape $D$ and $r \in \Z/n\Z$, we have $s^{(r)}_{D}(\q_1,\ldots,\q_m) \in \Q_\e^{\mS_m}$.
\end{thm}

\begin{example}\label{ex:loopcylindric}
Set $n=m=3$, $s=1$, and consider the shape $D$ in $\CC_2 = \Z^2/(2,1)\Z$ whose restriction to two adjacent columns is given by the
shape $\lambda = (2,1)$. 
The seven possible cylindric semistandard tableaux of shape $D$ are given by extending the tableaux from
\eqref{eq:21tableaux} periodically, except for the third tableau in that list.
Then we have 
$$
s_D^{(1)}(\q_1,\q_2,\q_3) =
q_1^{(1)} q_2^{(1)} (q_1^{(3)} + q_2^{(2)})
+ q_1^{(1)} q_3^{(3)} (q_1^{(3)} + q_2^{(2)}+q_3^{(1)})
+ q_2^{(3)} q_3^{(3)} (q_2^{(2)}+q_3^{(1)}). 
$$
\end{example}

The rest of this subsection is devoted to the proof of Theorem \ref{thm:q-cylindric-loop-schur}.  The idea is to write $s^{(r)}_{D}(\q_1,\ldots,\q_m)$ as a polynomial in the quantum loop elementary symmetric functions $e_k^{(r)} \in \Q_\e$.  There is some similarity to the proof of \cite[Theorem 3.5]{TP}.

%
%
%

For a highway path $P$ in $N_{n,m}$ we abuse notation by also writing $P$ for its weight $\wt(P)$.  If $\P = (P_1,P_2,\ldots,P_k)$ is an ordered family of (possibly intersecting) highway paths on $N_{n,m}$, its weight $\wt(\P)$ is the concatenation $P_1P_2 \cdots P_k$.    

For two monomials $m, m' \in \Q_\epsilon$, let us define $\alpha(m,m')$ by $mm' = \e^{\alpha(m,m')} m'm$.  Note that $\alpha(m',m) = - \alpha(m,m')$.  

Let $\tilde N_{n,m}$ denote the universal cover of $N_{n,m}$, where the vertices have weights $q_j^{(i)}$ arranged periodically.  Note that the snake paths (see \S\ref{ssec:Qe}) of $\tilde N_{n,m}$ are naturally indexed by $\Z$.  For a highway path $P$ on the universal cover $\tilde N_{n,m}$, we write $s(P) \subset \Z$ for the set of snake paths in $\tilde N_{n,m}$ that make a contribution to the weight $\wt(P)$.  For a highway path $P$ such that $s(P) \neq \emptyset$,
$s(P)$ is a set of consecutive numbers.
When $s(P) = [a,b]$, the path $P$ picks up one $q_{ji}$ on each snake path $S_k$ for $k \in [a,b]$, and the monomial $P$ has degree $b-a+1$. 
%

Define a function $\chi$ on $\Z$ by 
$$
\chi(z) = \chi_n(z) = \begin{cases} 1 & \mbox{if $z \equiv 0 \mod n$,} \\
0 & \mbox{otherwise.}
\end{cases}
$$
For a path $P$ and the weight $q$ of an interior vertex of $\tilde N_{n,m}$, 
we write $q \in P$ if $P$ picks up $q$.  Note that the weights of $\tilde N_{n,m}$ are now $q_j^{(i)}$ where $i \in \Z$, but the commutation relations for $q_j^{(i)}$ only depends on $i \in \Z/n\Z$.

\begin{lem}\label{lem:PP}
Let $P = P^{(1)} P^{(2)}$ be a decomposition of a highway path, where $s(P^{(1)}) = [a,b]$ and $s(P^{(2)}) = [b+1,c]$.  Then
$\alpha(P^{(1)},P^{(2)}) = 1 - \chi(a-b-1) - \chi(b-c) + \chi(a-c-1)$.
\end{lem}

\begin{proof}
Write $P^{(1)} = q_a q_{a+1}\cdots q_b$ and 
$P^{(2)} = q_{b+1} q_{b+2}\cdots q_c$, where $q_i$ is the weight of 
a vertex belonging
to the $i$-th snake path.  
For $q_i \in P^{(1)}$ and $q_j \in P^{(2)}$ we have
\begin{align}\label{eq:q-chi}  
\alpha(q_i,q_j) = \chi(i+1-j) - 2 \chi(i-j) + \chi(i-1-j), 
\end{align}
since on the cylinder the vertex with weight $q_j$ is to the right
of the vertex with weight $q_i$.
Therefore, $\alpha(P^{(1)},P^{(2)})$ is calculated by summing up 
$\alpha(q_i,q_j)$:
\begin{align*}
  \alpha(P^{(1)},P^{(2)})
  &= 
  \sum_{i=a}^b \sum_{j=b+1}^c \alpha(q_i,q_j)
  \\
  &=
  \chi(0) - \chi(a-b-1) - \chi(b-c) + \chi(a-c-1).
\end{align*}    
\end{proof}

We write $(j,i)$ for the intersection of the $i$-th
horizontal wire and the $j$-th vertical wire in $\tilde N_{n,m}$.
For $k \in \Z$ we define a map $m_k$ on the highway paths, 
which locally changes a highway path $P$ to $P'$ in the following way.
If $k \in s(P)$ and if there exists $j>1$ 
such that $q_{j}^{(k+1-j)} \in P$ and $q_{j-1}^{(k+2-j)} \notin P$,
remove the two edges $(j-1,k+2-j) \to (j-1,k+1-j) \to (j,k+1-j)$ from the path $P$ 
and add the two edges $(j-1,k+2-j) \to (j,k+2-j) \to (j,k+1-j)$.
Otherwise, $m_k$ does not change $P$.
Consequently we obtain $P'$ by replacing $q_{j}^{(k+1-j)}$ with $q_{j-1}^{(k+2-j)}$ in $P$. We remark that $q_{j}^{(k+1-j)}, q_{j-1}^{(k+2-j)} \in S_k$.

\begin{figure}[ht]
\unitlength=1.2mm
\begin{picture}(100,40)(0,0)

\multiput(5,5)(0,10){3}{\vector(1,0){30}}
\multiput(10,0)(10,0){3}{\vector(0,1){30}}

\put(8,31){\scriptsize $j-2$}
\put(18,31){\scriptsize $j-1$}
\put(30,31){\scriptsize $j$}
\put(-3,25){\scriptsize $k-j$}
\put(-7,15){\scriptsize $k-j+1$}
\put(-7,5){\scriptsize $k-j+2$}

\put(25,16){$P$}

{\color{red} 
\put(5,4){\line(1,0){16}}
\put(21,4){\line(0,1){10}}
\put(21,14){\vector(1,0){13}}
}
\put(42,15){$\stackrel{m_k}{\longrightarrow}$}

\multiput(65,5)(0,10){3}{\vector(1,0){30}}
\multiput(70,0)(10,0){3}{\vector(0,1){30}}

\put(68,31){\scriptsize $j-2$}
\put(78,31){\scriptsize $j-1$}
\put(90,31){\scriptsize $j$}
\put(57,25){\scriptsize $k-j$}
\put(53,15){\scriptsize $k-j+1$}
\put(53,5){\scriptsize $k-j+2$}

\put(85,6){$P'$}

{\color{red} 
\put(65,4){\line(1,0){26}}
\put(91,4){\line(0,1){10}}
\put(91,14){\vector(1,0){3}}
}

\end{picture}
    \caption{Action of $m_k$.}
\end{figure}
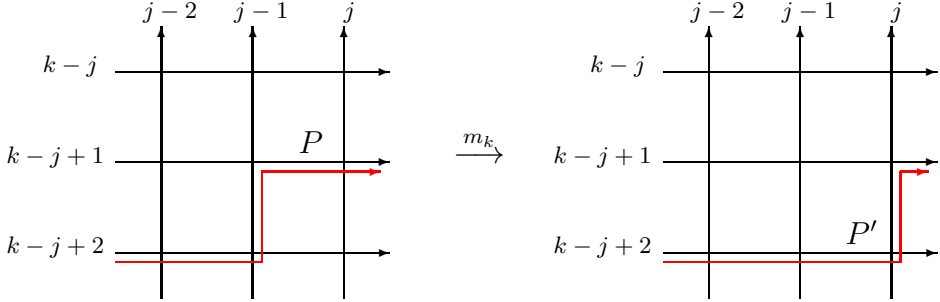

\begin{lem}\label{lem:m_k}
Let $P$ and $Q$ be highway paths on $\tilde N_{n,m}$ that do not intersect on $N_{n,m}$, with $s(P) = [a,b]$ and $s(Q) = [c,d]$.  Assume that $P$ and $Q$ have distinct endpoints on the same vertical line, but that $Q$ starts to the left of $P$.  Also assume that $c< a$.
Then, the following statements hold.
\begin{enumerate}
\item 
If $P$ and $m_k(Q)$ do not intersect on $N_{n,m}$, then $\alpha(P,Q) = \alpha(P,m_k(Q))$.
\item
If $m_k(P)$ and $Q$ do not intersect on $N_{n,m}$, then $\alpha(P,Q) = \alpha(m_k(P),Q)$.
\end{enumerate} 
\end{lem}

\begin{proof}
By translating $Q$ up or down $tn$ positions for some $t \in \Z$, we may assume that $Q$ lies above $P$ but is never more than $n$ positions above $P$.
In the same manner as in the proof of Lemma~\ref{lem:PP}, 
we write $P = p_a p_{a+1} \cdots p_b$ and 
$Q = q_c q_{c+1} \cdots q_d$.
For the weight $q^{(i)}_{j}$ at a crossing $(j,i)$, we write $h(q^{(i)}_{j}) = j$ for the horizontal position.

We prove (1).   
Assume that $k \in [c,d]$, and let $Q' = m_k(Q)$. Then $h(q_k') = h(q_k)-1$.
Note that  $\alpha(p_i,q_k) \neq 0$ holds only when $i \equiv k, k \pm 1 \mod n$.
We need to check that $\alpha(p_i,q_k) = \alpha(p_i,q_k')$ in these cases.
 
When $i \equiv k \mod n$, as $P$ does not intersect $Q$ and $Q'$ on $N_{n,m}$,
we have the following:
\begin{itemize}
\item if $k \in [a,b]$, then $h(p_k) < h(q_k)-1$,
\item for all $j > 0$ such that $k-jn \in [a,b]$, 
we have $h(p_{k-jn}) < h(q_k)-n-1$,
\item for all $j > 0$ such that $k+jn \in [a,b]$, 
we have $h(p_{k+jn}) > h(q_k)$,
\end{itemize}
From the first property, it follows that $h(p_k) < h(q_k')$, 
and we obtain $\alpha(p_k,q_k) = \alpha(p_k,q_k') = -2$.
From the second and third properties, 
it follows that $h(p_{k-jn}) < h(q_k')-n$ or $h(p_{k+jn}) > h(q_k')$,
and we obtain $\alpha(p_{k \mp jn},q_k) = \alpha(p_{k \mp jn},q_k') = \mp 2$.

When $i \equiv k \pm 1 \mod n$, in the same manner  
we obtain $\alpha(p_i,q_k) = \alpha(p_i,q_k')$ from 
\begin{itemize}
\item if $k+1 \in [a,b]$, then $h(p_{k+1}) < h(q_k)$,
\item for all $j > 0$ such that $k+1-jn \in [a,b]$, 
then $h(p_{k+1-jn}) < h(q_k)-n$,
\item for all $j > 0$ such that $k+1+jn \in [a,b]$, 
then $h(p_{k+1+jn}) > h(q_k)$,
 
\item if $k-1 \in [a,b]$, then $h(p_{k-1}) < h(q_k)-2$,
\item for all $j > 0$ such that $k-1-jn \in [a,b]$, 
then $h(p_{k-1-jn}) < h(q_k)-n-2$,
\item for all $j > 0$ such that $k-1+jn \in [a,b]$, 
then $h(p_{k-1+jn}) > h(q_k)$.
\end{itemize}

The proof of (2) is similar.
%
%
%
\end{proof}

\begin{lem}\label{lem:PQ}
Let $P$ and $Q$ be highway paths whose images in $N_{n,m}$ are non-intersecting, with $s(P) = [a,b]$ and $s(Q) = [c,d]$.  Assume that $P$ and $Q$ have distinct endpoints on the same vertical line, but that $Q$ starts to the left of $P$.  Assume that $c< a$.  
Then we have $\alpha(P,Q) =  - \chi(c-b-1) + \chi(c-a) +\chi(d-a+1)$.  
\end{lem}

\begin{proof}
By applying the $m_k$ for some $k \in [a,b] \cup [c,d]$, and using Lemma~\ref{lem:m_k}, we may assume that the paths $P$ and $Q$ are 
\begin{align*}
&P = q_{a-c}^{(c+1)} q_{a-c+1}^{(c+1)} \cdots q_{b-c+1}^{(c+1)},
\\
&Q = q_{1}^{(c)} q_{2}^{(c)} \cdots q_{d-c+1}^{(c)}. 
\end{align*}
We write $c(i,j)$ for the right-hand side of \eqref{eq:q-chi}.
For $q_{i-c}^{(c+1)} ~(i \in [a,b])$ and $q_{j-c+1}^{(c)}~(j \in [c,d])$ we have 
$$
  \alpha(q_{i-c}^{(c+1)}, q_{j-c+1}^{(c)})
  =
  \begin{cases}
    c(i,j) & j \in [i-1,d], \\
    -c(i,j) & j \in [c,i-2], 
  \end{cases} 
$$
and $\alpha(P,Q)$ is calculated as follows:
\begin{align*}
\alpha(P,Q) 
&= 
\sum_{i=a}^b \bigl( \sum_{j=i-1}^d c(i,j) - \sum_{j=c}^{i-2} c(i,j) \bigr)
\\
&=
\chi(a-d-1) - \chi(b-d) + \chi(a-c) - \chi(b-c+1) + 2 \chi(2) - 2 \chi(1)
\\
&=
\chi(a-d-1) + \chi(a-c) - \chi(b-c+1). \qedhere
\end{align*}
\end{proof}




\begin{proof}[Proof of Theorem~\ref{thm:q-cylindric-loop-schur}]
We first note that quantum cylindric loop Schur functions $s^{(\ell)}_{D}$ are exactly the highway measurements of $N_{n,m}$ (see \cite{LPS} for more details).  Let us fix source vertices $s_1 > s_2 > \cdots > s_r> s_1-n$ and sink vertices $t_1 > t_2 > \cdots > t_1-n$ on $\tilde N_{n,m}$.  Denote by $M = M(s_1,s_2,\ldots,s_r|t_1,t_2,\ldots,t_r)$ the highway measurement (as defined in \S\ref{ssec:highway}) consisting of families $\P = (P_1,\ldots,P_r)$ of non-intersecting highway paths in $N_{n,m}$ where $P_i$ lifts to a path on $\tilde N_{n,m}$ that goes from $\bar s_i$ (the image of $s_i$ in $N_{n,m}$) to $\bar t_i$.  This condition fixes the homology classes $[P_i]$.  Then $M$ is a quantum cylindric loop Schur function and every quantum cylindric loop Schur function occurs in this way.  For each $i \in \{1,2,\ldots,r\}$, define $a_i,b_i \in \Z$ by $s(P) = [a_i,b_i]$ where $P$ is a highway path in $\tilde N_{n,m}$ going from $s_i$ to $t_i$.

In the following, we obtain a canonical lifting to $ \tilde N_{n,m}$ of every path $P$ in $N_{n,m}$ considered by requiring that they start at one of the $s_i$.

Let $\P = (P_1,P_2,\ldots,P_r)$ be an ordered family of highway paths on $N_{n,m}$ where $P_i$ goes from $\bar s_i$ to some $\bar t_j$, such that each of $\bar t_1,\bar t_2,\ldots,\bar t_r$ is used once. 
Suppose $\P$ has intersections.  Among all pairs of paths $(P_i,P_j)$ that intersect, choose the pair where $j = i+1$ and $i$ is minimal.  Let $e$ be the right most edge on $P_i \cap P_j$.   Let $P_i = P_i^{(1)} P_i^{(2)}$ and $P_{i+1} = P_{i+1}^{(1)} P_{i+1}^{(2)}$ be the decompositions of the two paths, where $P_\ast^{(1)}$ contains $e$ and all edges in $P_\ast$ preceding it, while $P_\ast^{(2)}$ contains all edges after $e$.  Thus $P_i^{(2)}$ and $P_{i+1}^{(2)}$ do not intersect on $N_{n,m}$.   Let us suppose that $s(P_i^{(1)}) = [a,x-1]$, $s(P_i^{(2)}) = [x,b]$, $s(P_{i+1}^{(1)}) = [c,x'-1]$, and $s(P_{i+1}^{(2)}) = [x',d]$, where $x \equiv x' \mod n$.

%



Let $\P'$ be obtained from $\P$ by setting $P'_i = P_i^{(1)}P_{i+1}^{(2)}$ and $P'_{i+1} = P_{i+1}^{(1)}P_i^{(2)}$, and leaving the rest of the paths unchanged.  We compute $\alpha$ such that $\wt(\P) = \e^{\alpha} \wt(\P')$.  
By Lemma \ref{lem:PP} and \ref{lem:PQ}, we have
\begin{align}
\alpha &= \alpha(P_{i+1}^{(1)},P_{i+1}^{(2)}) + \alpha(P_i^{(2)},P_{i+1}) \nonumber \\ \label{eq:alpha}
&= (1-\chi(c - x') - \chi(x'-1-d) + \chi(c -d - 1)) \\ \nonumber 
& \qquad + (-\chi(c-b-1)+\chi(c-x)+\chi(d-x+1)) \\ 
&= 1+\chi(c-d-1)-\chi(c - b - 1).\nonumber 
\end{align}

Define 
$$\beta(\a,\b) = \beta((a_1,b_1),(a_2,b_2),\ldots,(a_r,b_r)) := \sum_{i < j \mid b_i < b_j} 
\left(-1-\chi(a_j-b_j-1)+\chi(a_j-b_i-1)\right)$$ 
and 
$\beta(\P) := \beta(s(P_1),\ldots,s(P_r))$.  Then \eqref{eq:alpha} gives
\begin{equation}\label{eq:PP}
\e^{\beta(\P)} \wt(\P) = \e^{\beta(\P')} \wt(\P').
\end{equation}

Let $e(a,b) = \sum_P P$ be the weight generating function of highway paths such that $s(P) = [a,b]$.  Then $\P \mapsto \P'$ sets up an involution on intersecting families of highway paths, and \eqref{eq:PP} allows us to cancel weights, giving
\begin{equation}\label{eq:M}
M = \sum_{\b'} (-1)^{\sign(\b')} \e^{\beta(\a,\b')-\beta(\a,\b)} e(a_1,b'_1) e(a_2,b'_2) \cdots e(a_k,b'_k)
\end{equation}
where the summation is over all $\b'$ such that a permutation of $\sigma(\b')$ is congruent to $\b$ modulo $n$\, and such that $\sum_i b_i = \sum_i b'_i$.  Here $\sign(\b') = \sign(\sigma)$.
The theorem follows from Corollary \ref{cor:elem} and the observation that $e(a,b)$ is nothing but a quantum loop elementary symmetric function $e^{(a)}_{b-a+1}$.
\end{proof}

\begin{example}
Let us consider the case of Example~\ref{ex:loopcylindric}.
Set $\a = (4,3)$ and $\b=(5,3)$. In \eqref{eq:M} we have three choices of 
$\b'$ as $(5,3), (3,5)$ and $(6,2)$, and obtain  
$$
  M(\a,\b) = e(4,5) e(3,3) - \e^{-2} e(4,3) e(3,5) - e(4,6) e(3,2).
$$
By using the formulae
\begin{align*}
&e(4,5) = e_2^{(4)} = q_1^{(1)} q_2^{(1)} + q_2^{(3)} q_3^{(3)} 
+ q_1^{(1)} q_3^{(3)}, 
\\ 
&e(3,3) = e_1^{(3)} = q_1^{(3)} + q_2^{(2)}+q_3^{(1)}, 
\\
&e(a,a+2) = e_3^{(a)} = q_1^{(i)} q_2^{(i)} q_3^{(i)}, \text{ for } i \equiv a \mod 3,
\end{align*}
and $e(a,a-1) = e^{(a)}_0=1$, 
we see that $M(\a,\b) = s_D^{(1)}(\q_1,\q_2,\q_3)$.
\end{example}

\section{Quantum cluster $R$-matrix} \label{sec:qclus}
As mentioned in the introduction, we will use Fock and Goncharov's quantization of
cluster $y$-seeds, and we begin by writing the cluster $R$-matrix in $y$-variables.

\subsection{Definition of the cluster $R$-matrix in $y$-variables}
Let $\mathbb{P}_{\rm univ}(\boldy)$ denote the universal semifield generated by $y_v$ for $v \in Q_{n,m}$. 
As in \S\ref{subsec:x-R}, for three adjacent closed vertical cycles $(M^-,M,M^+)$ 
in $Q_{n,m}$, we denote the $y$-variables $y_v$ for
$v \in M^-,M,M^+$ by $y_i^-, y_i, y_i^+$ respectively, where $i$ takes values modulo $n$ and the arrows around the vertex $i \in M$ are given by \eqref{eq:arrows}.  

\begin{thm}\label{thm:yRcluster}
Let $(M^-,M,M^+) = (M_{i-1},M_i,M_{i+1})$ denote three adjacent cycles in $Q_{n,m}$.  For any $j \in \Z/n \Z$, the operator $\mathcal{R}_{M,j}$ of \eqref{eq:R-op} acts on the $y$-seed $(Q_{n,m}, \boldy)$ by $\mathcal{R}_{M,j}(Q_{n,m}, \boldy) = (Q_{n,m}, \boldy')$, where 
\begin{align}
  \label{eq:yRcluster1}
  y_i' 
  &= (\alpha_{i+2})^{-1} \cdot y_{i+1}^{-1} \cdot \alpha_i,
  \\
  \label{eq:yRcluster2}
 y_{i^-}' 
  &= (\alpha_{i})^{-1} \cdot y_i \, y_i^- \cdot \alpha_{i+1},
  \\
  \label{eq:yRcluster3}
  y_{i^+}' 
  &= (\alpha_{i+1})^{-1} \cdot y_{i+1} \, y_i^+ \cdot \alpha_{i+2},
\end{align}
and the other $y_{v} ~({v} \notin M \cup M^+ \cup M^-)$ are unchanged.
%
Here we set
\begin{align}\label{eq:alpha-i}
\alpha_i =\alpha_i(M) := 1 + \sum_{k=1}^{n-1} y_i \, y_{i+1} \cdots 
y_{i+k-1} \in \mathbb{P}_{\rm univ}(\boldy).
\end{align}
The action of $\mathcal{R}_{M_0,j}$ (resp. $\mathcal{R}_{M_m,j}$) are given by \eqref{eq:yRcluster1} and
\eqref{eq:yRcluster3} with $M = M_0$ and $M^+ = M_1$ (resp. \eqref{eq:yRcluster1} and
\eqref{eq:yRcluster2} with $M = M_m$ and $M^- = M_{m-1}$).
\end{thm}

%
%
%

Theorem~\ref{thm:yRcluster} 
follows from its quantum version Theorem~\ref{thm:q-yRcluster}, stated in the next subsection.

We define the cluster $R$-matrix on $y$-variables as the rational morphism $R_{M_j} \colon \mathbb{P}_{\rm univ}(\boldy) \to \mathbb{P}_{\rm univ}(\boldy)$ 
for $j=0,\ldots,m$ given by 
Theorem~\ref{thm:yRcluster}.

\begin{example}\label{example:y}
When $n=3$, the formulae for $R_M(y_1)$, $R_M(y_1^-)$ and $R_M(y_1^+)$ are
\begin{align*}
R_M(y_1) &= (1 + y_3 + y_3 y_1)^{-1} y_2^{-1} 
            (1 + y_1 + y_1 y_2),
\\
R_M(y_1^-) &= (1 + y_1 + y_1 y_2)^{-1} y_1 y_1^-
                (1 + y_2 + y_2 y_3),
\\
R_M(y_1^+) &= (1 + y_2 + y_2 y_3)^{-1} y_2 y_1^+
                (1 + y_3 + y_3 y_1).
\end{align*} 
\end{example}



\subsection{Quantization}

In the following we study the quantization of $R_M$ 
using Fock-Goncharov quantization (see \S\ref{ssec:FG}) corresponding to the quiver $Q_{n,m}$. 
Let $\mathcal{Y}_\e := \mathcal{Y}_\e(Q_{n,m})$ be 
the ambient skew field of the quantum torus 
generated by
$y_{ji} ~(0 \leq j \leq m, ~i \in \Z/n\Z)$ with the relations:
\begin{align}
  y_{ji} y_{j+1,i} &= \e^2 y_{j+1,i} y_{ji} & 0 \leq j \leq m-1,
  \\
  y_{ji} y_{j,i+1} &= \e^{-2} y_{j,i+1} y_{ji} & 0 \leq j \leq m,
  \\
  y_{ji} y_{j+1,i-1} &= \e^{-2} y_{j+1,i-1} y_{ji} & 0 \leq j \leq m-1.
\end{align}

For any closed vertical cycle $M$ in $Q_{n,m}$ and $j \in \Z/n \Z$, we define a sequence of a permutation and 
$2n-2$ quantum mutations by 
$$
\mathcal{R}^\e_{M,j} 
:= s_{j-2,j-1} \circ 
(\mu_{j}^\e \circ \mu_{j+1}^\e \cdots \circ \mu_{j-4}^\e \circ \mu_{j-3}^\e) \circ 
(\mu_{j-1}^\e \circ \mu_{j-2}^\e \circ \cdots \circ \mu_{j+1}^\e \circ \mu_j^\e).
$$

\begin{thm}\label{thm:q-yRcluster}
Let $(M^-,M,M^+) = (M_{i-1},M_i,
M_{i+1})$ denote three adjacent cycles in $Q_{n,m}$.
Then, for any $j$ we have $\mathcal{R}^\e_{M,j}(Q_{n,m},\boldsymbol{y}) 
= (Q_{n,m},\boldsymbol{y}')$ where 
\begin{align}
  \label{eq:q-yRcluster1}
  y_i' 
  &= (\alpha_{i+2}^\e)^{-1} \cdot y_{i+1}^{-1} \cdot \alpha_i^\e,
  \\
  \label{eq:q-yRcluster2}
  y_{i^-}' 
  &= (\alpha_{i}^\e)^{-1} \cdot \e\, y_i \, y_i^- \cdot \alpha_{i+1}^\e,
  \\
  \label{eq:q-yRcluster3}
  y_{i^+}' 
  &= (\alpha_{i+1}^\e)^{-1} \cdot \e \,  y_{i+1} \, y_i^+ \cdot \alpha_{i+2}^\e,
\end{align}
and the other $y_{v} ~({v} \notin \cup M^- \cup M \cup M^+)$ are unchanged.
Here we set 
$$
  \alpha_i^\e = \alpha_i^\e(M) 
  := 1 + \sum_{k=1}^{n-1} \e^k y_i \, y_{i+1} \cdots y_{i+k-1} \in \Y_\e.
$$
The action of $\mathcal{R}_{M_0,j}^\e$ (resp. $\mathcal{R}_{M_m,j}^\e$) is given by \eqref{eq:q-yRcluster1} and
\eqref{eq:q-yRcluster3} with $M = M_0$ and $M^+ = M_1$ (resp. \eqref{eq:q-yRcluster1} and
\eqref{eq:q-yRcluster2} with $M = M_m$ and $M^- = M_{m-1}$).
\end{thm}

\begin{example}
We demonstrate the calculation of 
$\mathcal{R}_{M,1}^\e(Q_{n,m},\boldy) = (Q_{n,m},\boldy')$ in the case of 
$n=3$. Due to Proposition~\ref{prop:R-AB}
we have a sequence of seeds as 
$$
(Q_{3,m},\boldy) \stackrel{\mu_1^\e}{\mapsto} (Q_1,\boldy[1])
\stackrel{\mu_2^\e}{\mapsto} (Q_2,\boldy[2]) 
\stackrel{\mu_3^\e}{\mapsto} (Q_3,\boldy[3]) 
\stackrel{s_{2,3}}{\mapsto} (Q_1,\boldy[\bar 1]) 
\stackrel{\mu_1^\e}{\mapsto} (Q_{3,m},\boldy'). 
$$
In the following we write only $y$-variables $y_{v}$ of ${v} \in M^- \cup M \cup M^+$.
By starting with 
$$
  \boldy = 
  \begin{pmatrix}
  y_{1^-} & {\color{red}y_1} &y_{1^+}
  \\
  y_{2^-} & y_2 &y_{2^+}
  \\
  y_{3^-} & y_3 & y_{3^+}
  \end{pmatrix},
$$
we obtain 
{\small 
\begin{align*}
&\boldy[1] =
\begin{pmatrix}
y_1^-\frac{1}{1+\e y_1^{-1}} & y_1^{-1} & y_1^+(1+\e y_1) \\
y_2^-(1+\e y_1) & {\color{red}y_2 \frac{1}{1+\e y_1^{-1}}} & y_2^+ \\
y_3^- & y_3 (1+\e y_1) & y_3^+\frac{1}{1+\e y_1^{-1}}
\end{pmatrix},
\\[2mm]
&\boldy[2] =
\begin{pmatrix}
y_1^-\frac{1}{1+\e y_1^{-1}} & y_1^{-1}\frac{1}{1+\e {y_2[1]}^{-1}} & y_1^+(1+\e y_1) \\
y_2^-(1+\e y_1) & {y_2[1]}^{-1} & y_2^+(1+\e y_2[1]) \\
y_3^- (1+\e y_2[1]) & {\color{red}y_3 (1+\e y_1)} & 
y_3^+\frac{1}{1+\e y_1^{-1}}
\end{pmatrix},
\\[2mm]
&\boldy[3] =
\begin{pmatrix}
y_1^-\frac{1}{1+\e y_1^{-1}} & y_1^{-1}\frac{1}{1+\e {y_2[1]}^{-1}}(1+\e y_3[2]) 
& y_1^+(1+\e y_1) \\
y_2^-(1+\e y_1) & {y_2[1]}^{-1} & y_2^+(1+\e y_2[1])\frac{1}{1+\e y_3[2]^{-1}} \\
y_3^- (1+\e y_2[1])\frac{1}{1+\e y_3[2]^{-1}}& y_3[2]^{-1} & y_3^+\frac{1}{1+\e y_1^{-1}}
\end{pmatrix},
\\[2mm]
&\boldy[\bar 1] =
\begin{pmatrix}
y_1^-\frac{1}{1+\e y_1^{-1}} & {\color{red}y_1^{-1}\frac{1}{1+\e {y_2[1]}^{-1}}(1+\e y_3[2])} 
& y_1^+(1+\e y_1) \\
y_2^-(1+\e y_1) & {y_3[2]}^{-1} & y_2^+(1+\e y_2[1])\frac{1}{1+\e y_3[2]^{-1}} \\
y_3^- (1+\e y_2[1]) \frac{1}{1+\e y_3[2]^{-1}}& y_2[1]^{-1} & y_3^+\frac{1}{1+\e y_1^{-1}}
\end{pmatrix},
\\[2mm]
&\boldy' =
\begin{pmatrix}
y_1^-\frac{1}{1+\e y_1^{-1}}(1+\e y_1[3]) & y_1[3]^{-1} 
& y_1^+(1+\e y_1) \frac{1}{1+\e y_1[3]^{-1}}\\
y_2^-(1+\e y_1) \frac{1}{1+\e y_1[3]^{-1}} ~ & {y_3[2]}^{-1}(1+\e y_1[3]) & ~ y_2^+(1+\e y_2[1])\frac{1}{1+\e y_3[2]^{-1}} \\
y_3^- (1+\e y_2[1]) \frac{1}{1+\e y_3[2]^{-1}} ~& y_2[1]^{-1}\frac{1}{1+\e y_1[3]^{-1}}& ~y_3^+\frac{1}{1+\e y_1^{-1}}(1+\e y_1[3])
\end{pmatrix}.
\end{align*}
}
Here the $y$-variables corresponing to mutated vertices are indicated in red.
By arranging $\boldy'$, we obtain
\begin{align*}
R^\e_M(y_1) &= (1 + \e y_3 + \e^2 y_3 y_1)^{-1} y_2^{-1} 
            (1 + \e y_1 + \e^2 y_1 y_2),
\\
R^\e_M(y_1^-) &= (1 + \e y_1 + \e^2 y_1 y_2)^{-1} 
                  \e \, y_1 y_1^-
                (1 + \e y_2 + \e^2 y_2 y_3),
\\
R^\e_M(y_1^+) &= (1 + \e y_2 + \e^2 y_2 y_3)^{-1} 
                  \e \,y_2 y_1^+
                (1 + \e y_3 + \e^2 y_3 y_1).
\end{align*}
In the limit of $\e \to 1$, these reduce to Example~\ref{example:y}.
\end{example}


The following is the key lemma in the proof of Theorem~\ref{thm:q-yRcluster}.

\begin{lem}\label{lem:yRcluster-trop}
Assume that the $y$-variables lie in the tropical semifield 
$\mathbb{P}_{\rm trop}(\boldy)$.
Then for any $j \in \Z/n \Z$, the operator $\mathcal{R}_{M,j}$ defined at \eqref{eq:R-op} acts on the initial $y$-seed $(Q_{n,m}, \boldy)$ as 
$\mathcal{R}_{M,j}(Q_{n,m}, \boldy) = (Q_{n,m}, \boldy')$, where
\begin{align}\label{eq:yRcluster-trop}
  y'_{v} = 
  \begin{cases}
  (y_{i+1})^{-1}  
  & \mbox{if ${v} = i \in M$,}
  \\
  y_i \, y_i^-
  & \mbox{if $v = i^- \in M^-$,} 
  \\
  y_{i+1} \, y_i^+
  & \mbox{if $v = i^+ \in M^+$,}
  \\
  y_{v} & \mbox{otherwise.} 
  \end{cases}
\end{align}
\end{lem}

\begin{proof} 
Due to Proposition~\ref{prop:R-AB}, we compute $\mathcal{R}_{M,1}(Q_{n,m},\boldy)$ by the following sequence of $y$-seeds 
\begin{align}\label{eq:y-seq}
\begin{split}
(Q_0,\boldy[0]) &= (Q_{n,m},\boldy)
\stackrel{\mu_1}{\mapsto}
(Q_1,\boldy[1])
\stackrel{\mu_2}{\mapsto}
\cdots
\stackrel{\mu_{n}}{\mapsto}
(Q_{n},\boldy[n])
\stackrel{s_{n-1,n}}{\mapsto}
(Q_{n-2},\boldy[\overline{n-2}])
\\
&\qquad \stackrel{\mu_{n-2}}{\mapsto}
(Q_{n-3},\boldy[\overline{n-3}])
\stackrel{\mu_{n-3}}{\mapsto}
\cdots
\stackrel{\mu_{1}}{\mapsto}
(Q_0,\boldy[\overline{0}]) = \mathcal{R}_{M,1}(Q_{n,m},\boldy).
\end{split} 
\end{align}

Using the lemmas in \S\ref{subsec:Rcluster}, we explicitly calculate 
the change of tropical $y$-variables in \eqref{eq:y-seq}
by using the following general fact:
assume that a quiver $Q'$ has only simply weighted arrows,
and consider a tropical mutation $(Q'', \boldy'') = \mu_j(Q', \boldy')$
at $j \in Q'$.
If the $y$-variable at the mutation point, $y'_j$, 
is a nonnegative power of the initial $y$-variables,
then $\boldy'' = (y''_i)_i$ is determined by the following rule:
$$
  y''_i = 
  \begin{cases}
    y'_i y'_j & \text{if $j \to i$ in $Q'$},
    \\
    y'_i & \text{otherwise}.
  \end{cases} 
$$ 

In the following we write only $y_{v}[j]$ of 
${v} \in M^- \cup M \cup M^+$ as 
$$
  \boldy[0] = 
  \begin{pmatrix}
  y_{1^-} & {\color{red}y_1} &y_{1^+}
  \\
  \vdots & \vdots & \vdots
  \\
  y_{n^-} & y_n & y_{n^+}
  \end{pmatrix},
$$
where the mutation point in the sequence \eqref{eq:y-seq} is indicated in red.  Let $y_{[i,j]}:= y_i y_{i+1} \cdots y_j$.
By direct calculation we obtain the following,
where one sees that the `if' part of the above fact always holds 
in the sequence of mutations \eqref{eq:y-seq}.
\begin{align*}
  \boldy[k] &= 
  \begin{pmatrix}
  y_{1^-} y_1 & y_2 & y_{1^+}
  \\
  y_{2^-} & y_3 & y_{2^+}
  \\
  \vdots & \vdots & \vdots 
  \\
  y_{k-1^-} & y_k & y_{k-1^+}
  \\
  y_{k^-} & y_{[1,k]}^{-1} & y_{k^+}
  \\
  y_{k+1^-} & {\color{red}y_{[1,k+1]}} & y_{k+1^+}
  \\
  y_{k+2^-} & y_{k+2} & y_{k+2^+}
  \\
  \vdots & \vdots & \vdots 
  \\
  y_{n-1^-} & y_{n-1} & y_{n-1^+}
  \\
  y_{n^-} & y_{n} & y_{n^+} y_1
  \end{pmatrix}
  \qquad \text{for $k=1,2,\ldots,n-2$}, 
  \\[3mm] \displaybreak[0]
  \boldy[n-1] &=
  \begin{pmatrix}
  y_{1^-} y_1 & y_2 & y_{1^+}
  \\
  y_{2^-} & y_3 & y_{2^+}
  \\
  \vdots & \vdots & \vdots 
  \\
  y_{n-2^-} & y_{n-1} & y_{n-2^+}
  \\
  y_{n-1^-} & y_{[1,n-1]}^{-1} & y_{n-1^+}
  \\
  y_{n^-} & {\color{red}y_{n}} & y_{n^+} y_1
  \end{pmatrix},
  \qquad 
  \boldy[n] =
  \begin{pmatrix}
  y_{1^-} y_1 & y_2 & y_{1^+}
  \\
  y_{2^-} & y_3 & y_{2^+}
  \\
  \vdots & \vdots & \vdots 
  \\
  y_{n-2^-} & y_{n-1} & y_{n-2^+}
  \\
  y_{n-1^-} & y_{[1,n-1]}^{-1} & y_{n-1^+} y_n
  \\
  y_{n^-} y_n & y_{n}^{-1} & y_{n^+} y_1
  \end{pmatrix},
  \\[3mm] \displaybreak[0]
  \boldy[\overline{k}] &= 
  \begin{pmatrix}
  y_{1^-} y_1 & y_2 & y_{1^+}
  \\
  y_{2^-} & y_3 & y_{2^+}
  \\
  \vdots & \vdots & \vdots 
  \\
  y_{k^-} & {\color{red}y_{k+1}} & y_{k^+}
  \\
  y_{k+1^-} & y_{k+2}^{-1} & y_{k+1^+} y_{k+2}
  \\
  y_{k+2^-} y_{k+2} & y_{k+3}^{-1} & y_{k+2^+} y_{k+3}
  \\
  \vdots & \vdots & \vdots 
  \\
  y_{n-1^-} y_{n-1} & y_{n}^{-1} & y_{n-1^+} y_{n}
  \\
  y_{n^-} y_n & y_{[1,k+1]}^{-1} & y_{n^+} y_1
  \end{pmatrix}    
  \qquad \text{for $k=n-2,n-3,\ldots,0$.}
\end{align*}
From $\boldy[\overline{0}]$ we obtain \eqref{eq:yRcluster-trop} 
for the $j=1$ case, which is obviously independent of $j$.
\end{proof}

\begin{proof}[Proof of Theorem~\ref{thm:q-yRcluster}]
First we suppose that $M$ is not equal to either $M_0$ or $M_m$.
We use the notations $(Q_k, \boldy[k])$ and $(Q_k, \boldy[\overline{k}])$ 
defined in the proof of Lemma~\ref{lem:yRcluster-trop}.
From Theorem~\ref{thm:period} and Lemma~\ref{lem:yRcluster-trop},
it follows that $\mathcal{R}^\e_{M,j}(Q_{n,m},\boldsymbol{y})$
is independent of $j$. Therefore it is enough to check \eqref{eq:q-yRcluster1}--\eqref{eq:q-yRcluster3} for one $i$ of each. 
We explicitly calculate 
$y_{n^-}[\overline{0}]$, $y_{n-1^+}[\overline{0}]$ and $y_{n-1}[\overline{0}]$.

From the lemmas in \S\ref{subsec:Rcluster}, we obtain
\begin{align}
  \label{eq:yn-}
  y_{n^-}[\overline{0}] = y_{n^-}[n]
  &=
  y_{n^-} ( 1+\e \,y_{n-1}[n-2])(1+\e \,y_n[n-1]^{-1})^{-1},
  \\
  \label{eq:yn+}
  y_{n-1^+}[\overline{0}] = y_{n-1^+}[n]
  &= 
  y_{n-1^+} ( 1+\e \,y_{n-1}[n-2])(1+\e\, y_n[n-1]^{-1})^{-1},
  \\
  \label{eq:yn}
  y_{n-1}[\overline{0}] = y_{n-1}[\overline{n-3}]
  &=
  y_n[n-1]^{-1} \left( 1+\e \, y_{n-2}[\overline{n-2}] \right),
\end{align} 
and 
\begin{align}
  \label{eq:yj}
  y_{j+1}[j] &= y_{j+1} \left(1 + \e\,y_j[j-1]^{-1}  \right)^{-1}
  \quad j=1,\ldots,n-2,
  \\
  \label{eq:ynj}
  y_n[j] &= y_n[j-1](1+ \e\,y_{j}[j-1]) \quad j=1,\ldots,n-2,
  \\
  \label{eq:yn-2}
  y_{n-2}[[\overline{n-2}]= y_{n-2}[n] 
  &= 
  y_{n-2}[n-3]^{-1} ( 1+\e \,y_{n-1}[n-2]^{-1})^{-1}(1+\e \,y_n[n-1]).
\end{align}
By using \eqref{eq:yj} recursively, we obtain 
$$
  (1+\e \,y_1) (1+\e \,y_2[1]) \cdots (1 + \e\,y_{j}[j-1]) 
  =
  1 + \sum_{k=1}^{j} \e^k y_1 y_2 \cdots y_{k}, 
$$
for $j=1,\ldots,n-1$,
and it follows that
\begin{align}
  \label{eq:y-a1}
  y_{n-1}[n-2] &= \e^{n-2} (\alpha_1')^{-1} y_1 y_2 \cdots y_{n-1},
  \\
  \label{eq:y-a2}
  y_{n-2}[n-3] &= \e^{n-3} (\alpha_1'')^{-1} y_1 \cdots y_{n-2}.
\end{align}
Here we set 
$\alpha_1' = \alpha_1\e - \e^{n-1} y_1 \cdots y_{n-1}$,
$\alpha_1'' = \alpha_1' - \e^{n-2} y_1 \cdots y_{n-2}$. 
Furthermore, from \eqref{eq:ynj} and the fact that $y_n[n-1] = y_n[n-2]$ 
we obtain 
\begin{align}
\label{eq:y-a3}
  y_n[n-1] = y_n \alpha_1'.
\end{align}

Let us calculate $y_{n^-}[\overline{0}]$, $y_{n-1^+}[\overline{0}]$
and $y_{n-1}[\overline{0}]$.
The quiver $Q_{n-2}$ tells us (see \S\ref{ssec:FG}) that 
$y_n[n-2]$ and $y_{n-1}[n-2]$ commute and that 
$y_n[n-2] \cdot y_{n^-} = \e^{-2} y_{n^-} \cdot y_n[n-2]$. 
Thus, \eqref{eq:yn-} with \eqref{eq:y-a1} and \eqref{eq:y-a3} gives
\begin{align*}
  y_{n^-}[\overline{0}] 
  &=
  (1+\e^{-1} y_n[n-1]^{-1})^{-1} y_{n^-} ( 1+\e \,y_{n-1}[n-2]) 
  \\
  &=
  (1 + \e^{-1} (y_n \alpha_1')^{-1})^{-1} y_{n^-} (1+ \e^{n-1} (\alpha_1')^{-1} y_1 y_2 \cdots y_{n-1})
  \\
  &= 
  (\alpha_n^\e)^{-1} \cdot \e \, y_n y_{n^-} \cdot \alpha_1^\e, 
\end{align*}
where we have used that $\alpha_n^\e = 1 + \e \,y_n \alpha_1'$ and that
$\alpha_1'$ commutes with $y_{n^-}$.
In a similar way, from \eqref{eq:yn+} we obtain 
$y_{n-1^+}[\overline{0}] 
= (\alpha_n^\e)^{-1} \cdot \e \, y_n y_{n-1^+} \cdot \alpha_1^\e$.
From \eqref{eq:yn} and \eqref{eq:yn-2}, we get 
$$
  y_{n-1}[\overline{0}] 
  =
  \bigl[ 1+\e^{-1} (1+\e^{-1} y_{n-1}[n-2]^{-1})^{-1} y_{n-2}[n-3]^{-1} 
  (1+ \e\,y_n[n-2]) \bigr]
  y_n[n-2]^{-1}. 
$$
By substituting \eqref{eq:y-a1}--\eqref{eq:y-a3}, $y_{n-1}[\overline{0}]$   
is calculated to be
\begin{align*}
 &(y_n \alpha_1')^{-1} + \e^{-1}(\underline{1+\e^{-n+1}(y_1 \cdots y_{n-1})^{-1} 
 \alpha_1'})^{-1} \e^{-n+3} (y_1 \cdots y_{n-2})^{-1} \alpha_1''
 ((y_n \alpha_1')^{-1} + \e)
 \\
 &=(y_n \alpha_1')^{-1} + \e^{-1} (\alpha_1^\e)^{-1} \underline{\e^{n-1}(y_1 \cdots y_{n-1})
   \e^{-n+3} (y_1 \cdots y_{n-2})^{-1}} \alpha_1''((y_n \alpha_1')^{-1} + \e)
 \\
 &=
 (\alpha_1^\e)^{-1} \bigl[
  (\alpha_1' + \underline{\e^{n-1}y_1\cdots y_{n-1}})(y_n \alpha_1')^{-1} + \e^{-1} y_{n-1}
  \alpha_1''(\underline{(y_n \alpha_1')^{-1}} + \e) \bigr]
 \\
 &=
 (\alpha_1^\e)^{-1} \bigl[y_n^{-1} + \e^{-1} y_{n-1} \alpha_1' (y_n \alpha_1')^{-1} + y_{n-1} \alpha_1'' \bigr] 
 \\
  &= (\alpha_1^\e)^{-1} \bigl[y_n^{-1} + q^{-1} y_{n-1} y_n^{-1} + y_{n-1} \alpha_1'' \bigr]
  = (\alpha_1^\e)^{-1} y_n^{-1} \alpha_{n-1}^\e.
\end{align*}
Here the underlined parts denote key parts of the calculation.  This completes the proof for $M = M_k$ where $k = 1,2,\ldots,m-1$.

Finally, the last statement of the theorem follows from the following general fact: suppose $Q$ is a quiver and $\tilde Q \subset Q$ is the quiver restricted to some subset of vertices.  Let $A = \mu_{i_1} \circ \cdots \circ \mu_{i_r}$ be a sequence of mutations where $i_1,i_2,\ldots,i_r \in \tilde Q$.  Then $A(Q)$ restricted to the vertices in $\tilde Q$  is the same quiver as $A(\tilde Q)$.  We apply this to $M \cup M^+ \subset M^- \cup M \cup M^+$ (resp. $M^- \cup M \subset M^- \cup M \cup M^+$) to obtain the statement for $M = M_0$ (resp. $M = M_m$).
\end{proof}

\subsection{Yang-Baxter relation}
For $j=1,\ldots,m-1$ we define the quantum cluster $R$-matrix on $y$-variables 
as the morphism $R_{M_j}^{\e} \colon \mathcal{Y}_\e \to \mathcal{Y}_\e$ 
given by Theorem~\ref{thm:q-yRcluster}.

\begin{thm}\label{thm:q-YBR}
We have 
$R_{M_j}^\e R_{M_{j+1}}^\e R_{M_j}^\e = R_{M_{j+1}}^\e R_{M_j}^\e R_{M_{j+1}}^\e$ and $(R_j^\e)^2 = \id$. Furthermore, for $i,j$ such that $|i-j|>1$ we have
$R_{M_i}^\e R_{M_j}^\e = R_{M_j}^\e R_{M_i}^\e$.  Thus $R_{M_j}^\e$ generate an action of $\mS_{m}$ on $\mathcal{Y}_\e$.
\end{thm}

\begin{remark}
If we include $R_{M_0}^\e$ and $R_{M_m}^\e$, Theorem \ref{thm:q-YBR} extends to give an action of $\mS_{m+2}$.
\end{remark}
  
\begin{proof}
The equality $(R_j^\e)^2 = \id$ follows from the involutivity of quantum mutations.

Suppose $j \in \{1,\ldots,m-2\}$. Set $M=M_j$ and consider  
four adjacent cycles $M^-$, $M$, $M^+$ and $M^{++}$ 
in this order on $Q_{n,m}$. 
Due to Theorem~\ref{thm:period}, to prove $R_M R_{M^+} R_M = R_{M^+} R_{M} R_{M^+}$ it is enough to show 
that the relation $R^\e_M R^\e_{M^+} R^\e_M = R^\e_{M^+} R^\e_{M} R^\e_{M^+}$ hold for the initial tropical $y$-seed $(Q_{n,m},\boldy)$.

In the following, we explicitly write only $y_{v}$ of $v \in M^- \cup M \cup M^+ \cup M^{++}$, since the quiver $Q_{n,m}$ and the others $y$-variables are unchanged by $R_M$ and $R_{M^+}$. 
By using \eqref{eq:yRcluster-trop} we calculate $R_M R_{M^+} R_M(\boldy)$: 
\begin{align*}
  (y_{i^-},y_i,y_{i^+},y_{i^{++}})
  &\stackrel{R_M}{\mapsto}
  (y_{i^-} y_i, y_{i+1}^{-1}, y_{i+1} y_{i^+},y_{i^{++}})
  \\
  &\stackrel{R_{M^+}}{\mapsto}
   (y_i y_{i^-}, y_{i^+},(y_{i+2} y_{i+1^+})^{-1},
     y_{i+2} y_{i+1^+} y_{i^{++}})
  \\
  &\stackrel{R_M}{\mapsto}
  (y_i y_{i^-} y_{i^+}, y_{{i+1}^+}^{-1},y_{i+2}^{-1},
     y_{i+2} y_{i+1^+} y_{i^{++}}).
\end{align*}
Next, we calculate $R_{M^+} R_{M} R_{M^+}(\boldy)$:
\begin{align*}
  (y_{i^-},y_i,y_{i^+},y_{i^{++}})
  &\stackrel{R_{M^+}}{\mapsto}
  (y_{i^-}, y_i y_{i^+}, y_{{i+1}^+}^{-1},y_{{i+1}^+} y_{i^{++}})
  \\
  &\stackrel{R_{M}}{\mapsto}
  (y_{i^-} y_i y_{i^+}, (y_{i+1} y_{{i+1}^+})^{-1},
   y_{i+1},y_{{i+1}^+} y_{i^{++}})
  \\
  &\stackrel{R_{M^+}}{\mapsto}
  (y_{i^-} y_i y_{i^+}, y_{{i+1}^+}^{-1},y_{i+2}^{-1},
     y_{i+2} y_{i+1^+} y_{i^{++}}).
\end{align*}
Thus $R_M R_{M^+} R_M(\boldy) = R_{M^+} R_{M} R_{M^+}(\boldy)$ holds.
%

When $M$ and $M'$ are not adjacent, the identity
$R^\e_M R^\e_{M'} = R^\e_{M'}R^\e_M$ is immediate.
\end{proof}

\subsection{From quantum geometric $R$ to quantum cluster $R$}
In this section, we show that the quantum geometric $R$-matrix $R_j^\e$ and 
the quantum cluster $R$-matrix $R_{M_j}^{\e}$ are compatible.
In particular, we shall show that Theorem~\ref{thm:q-qYBR} follows from
Theorem~\ref{thm:q-YBR}. 

Let $\mathcal{Y}'_\e \subset \Y_\e$ be the skew subfield generated by $y_{ji}$ where $j \in 1,2,\ldots,m-1$.  We note that $\mathcal{Y}'_\e$ is closed under the operations $R_{M_1}^{\e},R_{M_2}^\e,\ldots,R_{M_{m-1}}^\e$.

\begin{prop}\label{prop:y-q}
The following transformation $\phi$ induces an embedding of 
skew fileds $\phi_\e \colon \mathcal{Y}'_\e \to \mathcal{Q}_\e(N_{n,m})$:
\begin{align}\label{eq:y-q}
y_{ji} \mapsto \e^{-1} \, q_{j,i}^{-1} q_{j+1,i-1} \qquad j=1,\ldots,m-1, ~i \in \Z/n\Z
\end{align}
\end{prop}

This proposition is easily proved by direct calculation.
We omit the proof.

\begin{thm}\label{thm:psi-RR}
For $j = 1,\ldots, m-1$, we have $R_j^\e \circ \phi_\e = \phi_\e \circ R_{M_{j}}^{\e}$.  Thus the quantum geometric $R$-matrix and 
the quantum cluster $R$-matrix are compatible. 
\end{thm}

\begin{proof}
Let $(W^{- -}, W^-, W, W^+) = (W_{j-1},W_{j},W_{j+1},W_{j+2})$ be 
four adjacent vertical cycles on $N_{n,m}$ 
where $q_i^{- -}$, $q_i^-$, $q_i$ and $q_i^+$ are located respectively,
and let $(M^-,M, M^+) = (M_{j-1},M_j,M_{j+1})$ be the corresponding three adjacent closed cycles on $Q_{n,m}$ 
where $y_i^-$, $y_i$ and $y_i^+$ are located respectively.   It is enough to show that $R_j^\e \circ \phi_\e(y_v)) = \phi_\e \circ R_{M_{j}}^{\e}(y_v)$ for $v \in M^- \cup M \cup M^+$.  

The map $\phi_\e$ sends the $y$-variables on $(M^-,M, M^+)$ to
$$
  y_i^- \mapsto \e^{-1} (q_{i}^{- -})^{-1} q_{i-1}^-, 
  \quad
  y_i \mapsto \e^{-1} (q_{i}^{-})^{-1} q_{i-1} ,
  \quad 
  y_i^+ \mapsto \e^{-1} (q_{i})^{-1} q_{i-1}^+ .
$$
We use the relation
$$
  \phi_\e(\alpha_{i+1}^\e(M)) = (q_{i-1}^- q_{i-2}^- \cdots q_{i-n+1}^-)^{-1} 
  \kappa_{i}^\e({\bf q^-},{\bf q}),
$$
which can be easily checked.  

We compute
\begin{align*}
\phi_\e(R_M^{\e}(y_i))  &= \phi_\e\left((\alpha_{i+2}^\e)^{-1} y_{i+1}^{-1} \alpha_i^\e \right) \\
&=(\kappa_{i+1}^\e)^{-1} q_{i}^- q_{i-1}^- \cdots q_{i-n+2}^- \cdot 
   \e q_{i}^{-1} q_{i+1}^{-} \cdot (q_{i-2}^- q_{i-3}^- \cdots q_{i-n}^-)^{-1} 
  \kappa_{i-1}^\e \\
&=(\kappa_{i+1}^\e)^{-1} \e^{-1} q_{i}^{-1} q_{i-1}^-
  \kappa_{i-1}^\e \\
&= \e^{-1} \cdot \left((\kappa_{i}^\e)^{-1}q_{i} \kappa_{i+1}^\e\right)^{-1} \cdot\left((\kappa_{i}^\e)^{-1} q_{i-1}^- \kappa_{i-1}^\e\right) \\
&= \e^{-1} R_j^\e(q_{i}^-)^{-1} \cdot R_j^\e(q_{i-1}) \\
&= R_j^\e(\phi_\e(y_i))
\end{align*}
and
\begin{align*}
\phi_\e(R_M^{\e}(y_i^+)) &= \phi_\e\left((\alpha_{i+1}^\e)^{-1} \cdot \e y_{i+1} y_i^+ \cdot \alpha_{i+2}^\e \right)\\
&=  (\kappa_{i}^\e)^{-1} q_{i-1}^- q_{i-2}^- \cdots q_{i-n+1}^-  \cdot  \e^{-1} (q_{i+1}^{-})^{-1} q_{i-1}^+ \cdot (q_{i}^- q_{i-1}^- \cdots q_{i-n+2}^-)^{-1} 
  \kappa_{i+1}^\e \\
&= \e^{-1} (\kappa_{i}^\e)^{-1}  (q_{i}^-)^{-1}  \kappa_{i+1}^\e q_{i-1}^+\\
&= \e^{-1} R_j^\e(q_{i})^{-1} \cdot q_{i-1}^+ \\
&= R_j^\e(\phi_\e(y_i^+)).
\end{align*}
The calculation for $y_i^-$ is similar.
\end{proof}

The following is an immediate consequence of Theorem~\ref{thm:psi-RR}.

\begin{cor}\label{cor:psi-RRclassical}
Let $\mathcal{Y}'$ be the universal semifield associated with the quiver
$Q_{n,m-1}$, generated by (commutative) 
$y$-variables $y_{ji}~(j=1,\ldots,m-1, ~i \in \Z/n\Z)$.
Let $\phi$ be the embedding map $\phi \colon \mathcal{Y}' \to \mathcal{Q}(N_{n,m})$
given by $y_{ji} \mapsto q_{j,i}^{-1} q_{j+1,i-1}$.
Then, the classical version of Theorem~\ref{thm:psi-RR} also holds:
for $j=1,\ldots,m-1$ we have 
$R_j \circ \phi = \phi \circ R_{M_j}$. 
\end{cor}

By combining Theorem~\ref{thm:iota-RR} and Corollary~\ref{cor:psi-RRclassical}, we get to the following:

\begin{thm}\label{thm:y-q-x}
We have a commutative diagram:
\begin{align*}
\xymatrix{
\mathcal{Y}' \ar[d]_{R_{M_j}} \ar[r]^{\phi \quad} & 
\mathcal{Q}(N_{n,m}) \ar[d]_{R_j}  
\ar[r]^{\iota_m} & \mathcal{F}(\tilde Q_{n,m}')   
\ar[d]^{\widetilde{R}_{M_{j}}}
\\
\mathcal{Y}' \ar[r]^{\phi ~~} & 
\mathcal{Q}(N_{n,m}) \ar[r]_{\iota_m} 
& \mathcal{F}(\tilde Q_{n,m}') \\
}
\end{align*}
for $j=1,\ldots,m-1$, where $\iota_m$ is given by \eqref{eq:xX-q}.
In particular,
we obtain the expression of $y_{ji}$ in $\mathcal{F}(\tilde Q_{n,m}')$,
\begin{align}\label{eq:y-xX}
  \iota_m \circ \phi (y_{ji})
  =
  \frac{x_{j-1,i+1} \, x_{j,i-1} \, x_{j+1,i}}{x_{j-1,i} \, x_{j,i+1} \, x_{j+1,i-1}} \cdot 
  \frac{X_{v_{(j,i),(j+1,i-1)}}}{X_{v_{(j-1,i+1),(j,i)}}}.
\end{align}
\end{thm}

We remark that the formula \eqref{eq:y-xX} can be written as follows:
for a non-frozen vertex $v$ in $\tilde Q'_{n,m}$ we have
\begin{equation}\label{eq:xy}
  \iota_m \circ \phi (y_{v}) = \prod_{u \in \tilde Q'_{n,m}} x_{u}^{b_{uv}}.
\end{equation}
Here $B = (b_{uv})$ is the exchange matrix corresponding to $\tilde Q'_{n,m}$, and in the right-hand side we take both frozen and non-frozen vertices for $u$.
See Figure~\ref{fig:clR5}.  Equation \eqref{eq:xy} should, for example, be compared to \cite[(11)]{FockGoncharov03}.

\subsection{Proof of Theorem~\ref{thm:commpres}}
\label{subsec:commpres}
Assume we are applying $R_j^{\e}$. We saw in the proof of Proposition \ref{prop:lens} that $q_{j,i+1}$ and $q_{j+1,i}$ satisfy the same commutation relations with any $q_{k,i}$ where $k \neq j,j+1$ 
as $(\kappa_{i+1}^{\e})^{-1}$ does. Let $q$ be such a $q_{k,i}$.  Recall that for monomials $m, m' \in \Q_{\e}$, we have defined the integer $\alpha(m,m')$ by
$m m' = \e^{\alpha(m,m')} m' m$.  We may use the same notation even when $m$ is not a monomial, as long as $m,m'$ commute up to a power of $\e$.

We see from the definition of $R_j^{\e}(q_{i,j+1})$ that 
\begin{align*}
\alpha(R_j^{\e}(q_{j+1,i}),q) 
&= \alpha((\kappa_{i+1}^{\e})^{-1},q) + \alpha(q_{i,j},q) + \alpha(\kappa_{i}^{\e},q)
\\
&= \alpha(q_{j+1,i},q) + \alpha(q_{j,i},q) - \alpha(q_{j,i},q) = \alpha(q_{j+1,i},q).
\end{align*}
We obtain $\alpha(R_j^{\e}(q_{j,i}),q) = \alpha(q_{j,i},q)$ in the same way.

It remains to argue that the commutation relations are preserved between the parameters in $\q_j$ and $\q_{j+1}$.

First we show $\alpha(q_{j,i},q_{j+1,i})$ is preserved.
Recall the notations in the proof of Proposition~\ref{prop:lens}, 
and write $p_i$ and $q_i$ for $q_{j,i}$ and $q_{j+1,i}$ respectively.
By a direct calculation we have 
\begin{align*}
&p_i \kappa_i^\e = \e^2 \kappa_i^\e p_i + (1-\e^2) p_{i-1} p_{i-2} \cdots p_i,
\\
&q_i \kappa_i^\e = \e^{-1} \kappa_i^\e q_i + (\e-\e^{-1}) p_{i-1} p_{i-2} \cdots p_{i+1} q_i. 
\end{align*}
Therefore we see that $q_i p_i \cdot \kappa_i^\e = \e \, \kappa_i^\e \cdot q_i p_i$.
By using this formula and \eqref{eq:pq-kappa}, we obtain
\begin{align*}
&p_i' q_i' = (\kappa_{i}^\e)^{-1} q_i p_i \kappa_i^\e = \e q_i p_i,
\\
&q_i' p_i' = (\kappa_{i+1}^\e)^{-1} p_i q_i \kappa_{i+1}^\e = \e^{-1} p_i q_i.
\end{align*}
Hence $\alpha(p_i',q_i') = \alpha(p_i,q_i) = 1$ follows.

Next we show that the other cases are preserved.
For elements $m, m' \in \Y_{\e}$, we define $\alpha(m,m')$ by
$m m' = \e^{\alpha(m,m')} m' m$, whenever $m,m'$ commute up to a power of $\e$.
Let $\boldy' = R_{M_{j}}^\e(\boldy)$.
We have $\alpha(y'_{h,i},y'_{j,k})=\alpha(y_{h,i},y_{j,k})$ since quantum mutations induce morphisms of quantum tori.

By Theorem~\ref{thm:psi-RR} and $\alpha(y'_{j,i},y'_{j+1,i})=\alpha(y_{j,i},y_{j+1,i})$ we obtain
\begin{align*}
&\alpha(q'_{j,i},q'_{j+1,i}) - \alpha(q'_{j,i},q'_{j+2,i-1})
- \alpha(q'_{j+1,i-1},q'_{j+1,i}) + \alpha(q'_{j+1,i-1},q'_{j+2,i-1}) 
\\
&=
\alpha(q_{j,i},q_{j+1,i}) - \alpha(q_{j,i},q_{j+2,i-1})
- \alpha(q_{j+1,i-1},q_{j+1,i}) + \alpha(q_{j+1,i-1},q_{j+2,i-1})
\end{align*}
assuming that all the parameters mentioned quasicommute.  But we have already shown that the first, second and fourth terms on each side of the equality are equal.  Thus it follows that $q'_{j+1,i-1}$ and $q'_{j+1,i}$ quasicommute and that $\alpha(q'_{j+1,i-1},q'_{j+1,i}) = \alpha(q_{j+1,i-1},q_{j+1,i})$.
In the same manner, from $\alpha(y'_{j-1,i},y'_{j,i}) = \alpha(y_{j-1,i},y_{j,i})$ 
it follows that $\alpha(q'_{j,i-1},q'_{j,i}) = \alpha(q_{j,i-1},q_{j,i})$.  This computes $\alpha(q,\tilde q)$ whenever $q$ and $\tilde q$ have distance one on $N_{n,m}$.

Next, we consider pairs of parameters that have distance two on $N_{n,m}$.  From $\alpha(y'_{j,i},y'_{j+1,i-1}) = \alpha(y_{j,i},y_{j+1,i-1})$ we obtain
$\alpha(q'_{j,i},q'_{j+1,i-1}) = \alpha(q_{j,i},q_{j+1,i-1})$
and 
from $\alpha(y'_{j-1,i},y'_{j+1,i})=\alpha(y_{j-1,i},y_{j+1,i})$ we obtain
$\alpha(q'_{j,i-1},q'_{j+1,i}) = \alpha(q_{j,i-1},q_{j+1,i})$.
Also, from
$\alpha(y'_{j,i},y'_{j-1,i-1}) = \alpha(y_{j,i},y_{j-1,i-1})$ we obtain
$\alpha(q'_{j,i},q'_{j,i-2}) = \alpha(q_{j,i},q_{j,i-2})$
and 
from $\alpha(y'_{j+1,i},y'_{j,i-1})=\alpha(y_{j+1,i},y_{j,i-1})$ we obtain
$\alpha(q'_{j+1,i},q'_{j+1,i-2}) = \alpha(q_{j+1,i},q_{j+1,i-2})$.
Repeating this argument we obtain the claim.

\subsection{Proof of Theorem~\ref{thm:q-qYBR}}
\label{subsec:YBRqR-yR}
From the Yang-Baxter relation for $y$-variables in Theorem~\ref{thm:q-YBR} we have $R_{M_j}^\e R_{M_{j+1}}^\e R_{M_j}^\e = R_{M_{j+1}}^\e R_{M_{j}}^\e R_{M_{j+1}}^\e$.  By Theorem~\ref{thm:psi-RR}, we deduce that
\begin{align}\label{eq:YBR-psi}
 R_j^\e R_{j+1}^\e R_j^\e  \circ \phi_\e
  =  R_{j+1}^\e  R_{j}^\e R_{j+1}^\e \circ \phi_\e.
\end{align}
Thus $R_j^\e R_{j+1}^\e R_j^\e =  R_{j+1}^\e  R_{j}^\e R_{j+1}^\e$ holds on the image of $\phi_\e$, and it holds trivially when acting on $\q_1,\ldots,\q_{j-1},\q_{j+3},\ldots,\q_m$.  Assuming that $2 \leq j\leq m-3$, it follows easily from the definition of $\phi_\e$ that $R_j^\e R_{j+1}^\e R_j^\e =  R_{j+1}^\e  R_{j}^\e R_{j+1}^\e$ holds on $\Q_\e$.  Since this identity is `local' in the sense that $R_j^\e$ and $R_{j+1}^\e$ act only on $\q_j,\q_{j+1},\q_{j+2}$, we see that the identity holds for all $j$.

This establishes the braid relations.  The involutivity of $R_j^\e$ is proved in a similar manner, and the equality $R_i^\e R_j^\e = R_j^\e R_i^\e$ for $|i-j|>1$ is clear.

\section{Proofs} \label{sec:proof}
\subsection{Proof of Theorem \ref{thm:Rcluster}}\label{subsec:Rcluster}
Let $A$ denote the sequence of mutations at $1,2,\ldots,n-2$ in that order.  Let $A_i$ denote the first $i$ steps of these mutations.  Let $B$ denote the mutations $n-1, n$ followed by the permutation $s_{n-1,n}$.  Then the sequence of mutations in $\mR_{M,1}$ is exactly $A^{-1} \circ B \circ A$.

Let $Q$ be the initial quiver, restricted to $M^- \cup M \cup M^+$.
Let us now describe $Q_i := A_i(Q)$, and set $Q' = Q_{n-2} = A(Q)$.  This is described in three levels.  We first describe the quiver restricted to $M$.  Then we describe all arrows between $M$ and $M^\pm$.  Lastly, we describe what happens to arrows within $M^+ \cup M^-$.

The following results are proved by a straightforward induction.
\begin{lem}
For $i < n-2$, the quiver $Q_i|_M$ consists of:
\begin{enumerate}
\item
a possibly empty directed path $1 \to 2 \to \cdots \to i$,
\item
an oriented triangle $i \to n \to i+1 \to i$,
\item
an oriented $(n-i)$-gon $i+1 \to \cdots \to n \to i+1$.
\end{enumerate}
The edge $n \to i+1$ is mentioned twice, but it only occurs once.
\end{lem}

\begin{lem}
The quiver $Q'|_M$ consists of
\begin{enumerate}
\item
the directed path $1 \to 2 \to \cdots \to n-2$, 
\item
the edges $n-2 \to n$ and $n-1 \to n-2$.
\end{enumerate}
\end{lem}

\begin{lem}
Suppose $i \geq 1$.  The edges of $Q_i$ between $M$ and $M^-$ are:
\begin{enumerate}
\item
$1^- \to 1$,
\item
$(1 \to 2^-), (2 \to 3^-), \ldots, ((i-1) \to i^-), (i \to (i+1)^-)$,
\item
$(3^- \to 1), (4^- \to 2), \ldots, ((i+1)^- \to i-1)$, if $i \geq 2$,
\item 
$((i+2) \to (i+2)^-), ((i+3) \to (i+3)^-), \ldots, (n \to n^-)$,
\item
$((i+3)^- \to i+2), ((i+4)^- \to i+3), \ldots, (n^- \to (n-1))$.
\end{enumerate}
\end{lem}

\begin{lem}
Suppose $i \geq 1$.  The edges of $Q_i$ between $M$ and $M^+$ are:
\begin{enumerate}
\item
$(n^+\to 1)$,
\item
$(1 \to 1^+), (2 \to 2^+),\ldots, (i \to i^+)$,
\item
$(2^+ \to 1), (3^+ \to 2), \ldots, (i^+ \to (i-1))$,
\item 
$((i+2) \to (i+1)^+), ((i+3) \to (i+2)^+) , \ldots, (n^+ \to (n-1))$,
\item
$((i+1)^+ \to (i+1)), ((i+2)^+ \to (i+2)), \ldots, ((n-1)^+ \to (n-1))$.
\end{enumerate}
\end{lem}

\begin{lem}
Suppose $i \geq 1$.  The edges of $Q_i$ within $M^-$ are:
\begin{enumerate}
\item The directed path $2^- \to 3^- \to \cdots \to n^- \to 1^-$.
\end{enumerate}
\end{lem}

\begin{lem}
Suppose $i \geq 1$.  The edges of $Q_i$ within $M^+$ are:
\begin{enumerate}
\item The directed path $1^+ \to 2^+ \to \cdots \to n^+$.
\end{enumerate}
\end{lem}

\begin{lem}
Suppose $i \geq 1$.  The edges of $Q_i$ between $M^-$ and $M^+$ are:
\begin{enumerate}
\item $(1^+ \to 1^-)$,
\item $(2^- \to n^+)$.
\end{enumerate}
\end{lem}

\begin{prop}\label{prop:R-AB}
We have $B(Q') = Q'$.  Thus $A^{-1} \circ B \circ A$ sends the quiver $Q$ to itself.
\end{prop}

We now prove that $\mR_{M,1}$ acting on the variables $x_i$ is given by the formula $R_M$ from \eqref{eq:Rx}.

\begin{prop}
The effect of $A$ on the cluster variables is given by the formula $x_i \mapsto \tilde x_i$, where
\begin{align}\label{eq:half}
\begin{split}
\tx_i &= \frac{1}{x_1x_2\cdots x_i} \left(x_1^-(x_2 x_3 \cdots x_{i+1})x_n^+ + \sum_{k=3}^{i+2} x_{k-1}^- (x_k x_{k+1} \cdots x_{i+1} x_n x_1 \cdots x_{k-2}) x_{k-2}^+\right) \\
&=  x_1^-\frac{x_{i+1}}{x_1} x_n^+ + \sum_{k=3}^{i+2} x_{k-1}^- \frac{x_{i+1} x_n}{x_{k-1}x_{k-2}} x_{k-2}^+
\end{split}
\end{align}
for $i =1,2,\ldots,n-2$.  Note that there are $(i+1)$ terms in this formula.
\end{prop}

It is easy to check that $B$ acting on the variables $\tx_1,\tx_2,\ldots,\tx_{n-2},x_{n-1},x_n$ sends $x_{n-1}$ and $x_n$ to $R_M(x_{n-1})$ and $R_M(x_n)$.  Since all mutations act as invertible rational transformations, to prove that $\mR_{M,1} = R_M$, it suffices to check that $A \circ R_M = B \circ A$, where $A$ and $B$ are thought of as operators acting on an ordered sequence of $n$ variables.  In other words, we need to compute $\tx_i(R_M(x_1),\ldots,R_M(x_n))$.  

Let us substitute the formula for $R_M$ into \eqref{eq:half}.  Note that $R_M(x_i) = Sx_i$ for some expression $S$.  Since the formula for $\tx_i$ is homogeneous in $x_1,x_2,\ldots,x_n$, we see that $\tx_i(R_M(x_1),\ldots,R_M(x_n)) = \tx_i$.  This completes the proof of Theorem \ref{thm:Rcluster}.

\subsection{Proof of Theorem \ref{thm:Rbraidsimple}}
\label{subsec:Rbraidsimple}
Let 
\begin{equation}\label{eq:S}
S = \frac{\sum_{j=1}^n x_{j+1}^- 
            \left(\prod_{\ell = j+2}^{j-1} x_{\ell} \right) x_j^+}
           {\prod_{j =1}^n x_j}.
\end{equation}  Then $R_M(x_i) = Sx_i$.  Note that $R_M(S) = S(R_M(x_1),\ldots,R_M(x_n)) = S/S^2 = 1/S$.  Thus $$R_M^2(x_i) = S(R_M(x_1),\ldots,R_M(x_n)) R_M(x_i) = x_i.$$

\label{sec:Rbraid}
We have $R_{M^+}(S) = S|_{x_i^+ \mapsto S^+ x_i^+} = S^+ S$ since $S$ is homogeneous of degree one in the $+$-variables.  Similarly $R_{M}(S^+) = S S^+$.  Thus
$$
R_M R_{M^+} R_M(x_i) = R_M R_{M^+} (S x_i) = R_M(S^+ S x_i) = (SS^+)(1/S) Sx_i = SS^+x_i.
$$
Also,
$$
R_{M^+} R_M R_{M^+} (x_i) = R_{M^+}R_M  (x_i) = R_{M^+}(S x_i) =SS^+x_i.
$$
Similarly, $R_M R_{M^+} R_M(x_i^+) = R_{M^+} R_M R_{M^+} (x_i^+)$.  We conclude that $R_M R_{M^+} R_M = R_{M^+} R_M R_{M^+}$.

When $M$ and $M'$ are nonadjacent closed vertical cycles, the equality $R_M R_{M'} = R_{M'} R_M$ is clear.

\subsection{Proof of Theorem~\ref{thm:Rcluster2}}
\label{subsec:Rcluster2}

Let $\tilde{Q}$ be the subquiver of $\tilde Q_{n,m}$, 
restricted to $M^- \cup M \cup M^+$ together with all the frozen vertices
neighbouring $M$. 
Let $Q_L$ be the quiver obtained from $\tilde{Q}$ by removing
the vertices corresponding to the frozen variables
$X_{i^-,i}$, $X_{i,i^+}$, $X_{i+1^-,i}$ and $X_{i+1,i^+}$ for $i=1,\ldots,n$.
In the same manner, let $Q_C$ be the quiver without the vertices corresponding to the frozen variables
$X_{i,i+1}$, $X_{i+1^-,i}$ and $X_{i+1,i^+}$,
and let $Q_D$ be the quiver without the vertices corresponding to the frozen variables
$X_{i,i+1}$, $X_{i^-,i}$ and $X_{i,i^+}$.
We write $\bx_L$, $\bx_C$ and $\bx_D$ 
for the clusters (including the remaining frozen variables)
corresponding to the quivers $Q_L$, $Q_C$ and $Q_D$ respectively.
Similar to the proof in \S\ref {subsec:Rcluster}, we decompose $\tcR_{M,1} = A^{-1} \circ B \circ A$.
The following lemma is easily seen.

\begin{lem}\label{lem:sigma-op}
Let $\sigma\colon (\tilde{Q},\bx) \mapsto \sigma(\tilde{Q},\bx)$ be the operation that acts on the mutable
vertices by
$i^+ \mapsto i+1^-$, $i^- \mapsto i-1^+$, and acts on the frozen vertices by
$v_{i^-,i} \mapsto v_{i,i-1^+}$, $v_{i,i-1^+} \mapsto v_{i^-,i}$, 
$v_{i,i^+} \mapsto v_{i+1^-,i}$, $v_{i+1^-,i} \mapsto v_{i,i^+}$
for all $i=1,\ldots,n$, fixing all other vertices.  The action on the cluster variables $\bx$ is the natural induced one.  Then we have $\sigma(\tilde{Q},\bx) = (\tilde Q,\bx)$,
and $\sigma(Q_C,\bx_C) = (Q_D,\bx_D)$.   
\end{lem}

Quiver mutation produces no new arrows between frozen vertices.  Thus the contribution of the frozen vertices in a sequence of mutations of $(\tilde{Q}, \bx)$ is the superposition of the three distinct 
contributions in 
the corresponding mutations of $(Q_L,\bx_L)$, $(Q_C,\bx_C)$ and $(Q_D,\bx_D)$.

\begin{prop}
We have 
$\tcR_{M,1}(Q_\ast) = Q_\ast$ for $\ast = L,C$ and $D$.
\end{prop}  
 
\begin{proof}
We calculate the change of arrows related to the frozen vertices, caused by 
$\tcR_{M,1}$, in addition to the change of $Q$ studied in 
\S\ref{subsec:Rcluster}.

The case of $\ast = L$:
for $i=1,\ldots,n-2$, the arrows related to $X_{i,i+1}$ are changed 
only by $\mu_i$ in $A$:
\begin{align*}
(i+1 \to X_{i,i+1} \to i) \stackrel{\mu_i}{\mapsto}
  \begin{cases}
    (1 \to X_{1,2}), (X_{1,2} \to 1^-), (X_{1,2},n^+) & i=1,
    \\
    (i \to X_{i,i+1} \to i-1) & i=2,\ldots,n-2,     
  \end{cases} 
\end{align*}
and the resulted arrows are invariant under $B$.
For $i=n-1$, the arrows related to $X_{n-1,n}$ are not affected by $A$,
and invariant under $B$:  
\begin{align*}
&(n \to X_{n-1,n} \to n-1) \stackrel{\mu_{n-1}}{\mapsto} 
(n-1 \to X_{n-1,n}), (n \to X_{n-1,n}) \\
&\stackrel{\mu_{n}}{\mapsto}
(n-1 \to X_{n-1,n} \to n) \stackrel{s_{n-1,n}}{\mapsto}
(n \to X_{n-1,n} \to n-1).
\end{align*}  
For $i=n$, the arrows related to $X_{n,1}$ are changed only by $\mu_1$ in $A$:
$$
(1 \to X_{n,1} \to n) \stackrel{\mu_1}{\mapsto}
(1^+,X_{n,1}),(2^- \to X_{n,1}),(X_{n,1} \to 1),
$$ 
and invariant under $B$.
Then we have $B \circ A (Q_L) = A(Q_L)$.

The case of $\ast = C$:
for $i=1,\ldots,n-2$, the arrows related to $X_{i,i^+}$ are changed 
only by $\mu_i$ in $A$:
\begin{align}\label{eq:C+1}
(i \to X_{i,i^+} \to i^+) \stackrel{\mu_i}{\mapsto}
    (n \to X_{i,i^+}), (i+1^- \to X_{i,i^+}),(X_{i,i^+} \to i).
\end{align}
It is further changed by $B$ into
\begin{align}\label{eq:C+41}
(i+1^- \to X_{i,i^+}), (n-2 \to X_{i,i^+}),
(X_{i,i^+} \to i), (X_{i,i^+} \to n-1),
\end{align}
for $i=1,\ldots,n-3$, and 
\begin{align}\label{eq:C+42}
(n-1^- \to X_{n-2,n-2^+} \to n-1), 
\end{align}
for $i=n-2$.
For $i=n,n-1$, the arrows related to $X_{i,i^+}$ are not affected by $A$
but by $\mu_i$ and $s_{n-1,n}$ in $B$ as  
\begin{align}
&(n-1 \to X_{n-1,n-1^+} \to n-1^+) 
\stackrel{\mu_{n-1}}{\mapsto} 
(X_{n-1,n-1^+} \to n-1), (n^- \to X_{n-1,n-1^+}) \nonumber \\ 
& \quad \stackrel{s_{n-1,n} \circ \mu_{n}}{\mapsto}
(n^- \to X_{n-1,n-1^+} \to n).
\label{eq:C+2}
\\
&(n \to X_{n,n^+} \to n^+) 
\stackrel{\mu_{n}}{\mapsto} 
(X_{n,n^+} \to n), (X_{n,n^+} \to n^+),(n-2 \to X_{n,n^+}) \nonumber \\
& \quad \stackrel{s_{n-1,n}}{\mapsto}
(X_{n,n^+} \to n-1), (X_{n,n^+} \to n^+),(n-2 \to X_{n,n^+}).  
\label{eq:C+3}
\end{align}  

The arrows related to $X_{i^-,i} ~(i=1,\ldots,n-2)$ are changed by 
$\mu_i, \mu_{i+1},\ldots, \mu_{n-2}$ in $A$.
For $i=1$ and $j=1,\ldots,n-2$ we have 
\begin{align}\label{eq:C-11}
(1^- \to X_{1^-,1} \to 1) \stackrel{\mu_j \circ \cdots \circ \mu_1}{\mapsto}
(j \to X_{1^-,1}),(X_{1^-,1} \to j+1),(X_{1^-,1} \to n^+), 
\end{align}
and for $i=2,\ldots,n-2$ and $j=i,\ldots,n-2$ we have 
\begin{align}\label{eq:C-12}
(i^- \to X_{i^-,i} \to i) \stackrel{\mu_j \circ \cdots \circ \mu_i}{\mapsto}
(i^- \to X_{i^-,i}),(j \to X_{i^-,i}),(X_{i^-,i} \to i-1),(X_{i^-,i} \to j+1). 
\end{align}
Furthermore, the action of $B$ changes \eqref{eq:C-11} into
\begin{align}\label{eq:C-13}
(n \to X_{1^-,1} \to n^+)
\end{align}
and \eqref{eq:C-12} of $j=n-2$ is changed into 
\begin{align}\label{eq:C-14}
(i^- \to X_{i^-,i}),(n \to X_{i^-,i}),(X_{i^-,i} \to i-1).
\end{align}
The arrows related to $X_{i^-,i} ~(i=n-1,n)$ are not affected by $A$
but by $\mu_i$ and $s_{n-1,n}$ in $B$:
\begin{align}
&(n-1^- \to X_{n-1^-,n-1} \to n-1) \nonumber \\
&\quad \stackrel{\mu_{n-1}}{\mapsto} 
(n-1^- \to X_{n-1^-,n-1}),(n-1,X_{n-1^-,n-1}),(X_{n-1^-,n-1},n-2) 
\nonumber \\
& \quad \stackrel{s_{n-1,n} \circ \mu_{n}}{\mapsto}
(n-1^- \to X_{n-1^-,n-1}),(n,X_{n-1^-,n-1}),(X_{n-1^-,n-1},n-2),
\label{eq:C-21} 
\\
&(n^- \to X_{n^-,n} \to n) \stackrel{\mu_{n}}{\mapsto}    
(n \to X_{n^-,n} \to n-1^+) 
\stackrel{s_{n,n-1}}{\mapsto} (n-1 \to X_{n^-,n} \to n-1^+).      
\label{eq:C-22} 
\end{align}

We can check $B \circ A (Q_C) = A(Q_C)$ by comparing 
\eqref{eq:C+1}--\eqref{eq:C+3} and \eqref{eq:C-11}--\eqref{eq:C-22}
in the following way.   

The arrows \eqref{eq:C-14} with $i=2$ in $B \circ A (Q_C)$ corresponds to
the last part of \eqref{eq:C+1} with $i=1$ in $A(Q_C)$. 
Similarly, \eqref{eq:C-14} $(i=3,\ldots,n-2)$ and  
the last part of \eqref{eq:C-21} in $B \circ A (Q_C)$ 
respectively corresponds to 
the last part of \eqref{eq:C+1} $(i=2,\ldots,n-2)$ in $A (Q_C)$.
Eq.~\eqref{eq:C-13} corresponds to the first part of \eqref{eq:C+3},
and \eqref{eq:C-22} corresponds to the first part of \eqref{eq:C+2}.
Here we see that the frozen variable $X_{i^-,i}$ appearing in 
$B \circ A (Q_C)$ corresponds to $X_{i-1,i-1^+}$ appearing in $A(Q_C)$,
from which it follows that $\tcR_{M,1}$ replaces $X_{i-1,i-1^+}$
with $X_{i^-,i}$ in $\bx_L$.

Moreover, \eqref{eq:C+41} $(i=1,\ldots,n-3)$ in $B \circ A (Q_C)$ respectively
corresponds to the last part of \eqref{eq:C-12} $(i=2,\ldots,n-2)$ in $A(Q_C)$
with $j=n-2$. Eq.~\eqref{eq:C+42} and \eqref{eq:C+2} in $B \circ A (Q_C)$
respectively corresponds to the first part of \eqref{eq:C-21} and 
\eqref{eq:C-22} in $B (Q_C)$.
The last part of \eqref{eq:C+3} in $B \circ A(Q_C)$ corresponds to
that of \eqref{eq:C-11} with $j=n-2$  in $A(Q_C)$.    
Here we see that the frozen variable $X_{i,i^+}$ appearing in 
$B \circ A (Q_C)$ corresponds to $X_{i+1^-,i+1}$ appearing in $A(Q_C)$,
from which it follows that $\tcR_{M,1}$ replaces $X_{i+1^-,i+1}$
with $X_{i,i^+}$.

The case of $\ast = D$ follows from the case of $\ast = C$ and 
Lemma~\ref{lem:sigma-op}.
\end{proof}

From the change of arrows studied in the above proof,
the following lemmas are obtained.

\begin{lem}\label{lem:x-L}
For the cluster variables $x_i$ in $\bx_L$, we have that $\tx_i = A(x_i)$ is given by
\begin{align*}
\tx_i 
=  
\sum_{k=0}^{i} x_{k+1}^- \frac{x_{i+1} x_n}{x_{k+1}x_{k}} x_{k}^+
X_{k,k+1}; \quad i =1,2,\ldots,n-2.
\end{align*}
Furthermore, $B$ changes $x_{n-1}$ and $x_n$ into
\begin{align*}
&\tx_{n-1} = \frac{\tx_{n-2}}{x_{n}} + \frac{x_{n}^- x_{n-1}^+}{x_{n}}
              X_{n-1,n}, 
\\
&\tx_{n} = \frac{\tx_{n-2}}{x_{n-1}} + \frac{x_{n}^- x_{n-1}^+}{x_{n-1}}
              X_{n-1,n}.
\end{align*}
\end{lem}

\begin{lem}
For the cluster variables $x_i$ in $\bx_C$, we have that $\tx_i = A(x_i)$ is given by
\begin{align}
\label{eq:RC-xi}
\tx_i 
=
\sum_{k=0}^{i} x_{k+1}^- \frac{x_{i+1} x_n}{x_{k+1}x_{k}} x_{k}^+
\prod_{\ell=1}^k X_{\ell^-,\ell} \prod_{\ell=k+1}^i X_{\ell,\ell^+}
; \quad i =1,2,\ldots,n-2.
\end{align}
Furthermore, $B$ changes $x_{n-1}$ and $x_n$ into 
\begin{align}
\label{eq:RC-xn-1}
&\tx_{n-1} 
= 
\frac{\tx_{n-2}}{x_{n}} X_{n^-,n} 
+ \frac{x_{n}^- x_{n-1}^+}{x_{n}} \prod_{\ell \neq n-1} X_{\ell,\ell^+},
\\
\label{eq:RC-xn}
&\tx_{n} 
= 
\frac{\tx_{n-2}}{x_{n-1}} X_{n-1^-,n-1} 
+ \frac{x_{n}^- x_{n-1}^+}{x_{n-1}} \prod_{\ell \neq n} X_{\ell^-,\ell}.
\end{align}
\end{lem}

\begin{proof}
From \eqref{eq:C+1} and \eqref{eq:C-11}, we recursively obtain
\begin{align}
  &\tx_1 = \frac{x_1^- x_2 x_n^+ X_{1,1^+}}{x_1} 
          + \frac{x_2^- x_n x_1^+ X_{1^-,1}}{x_1},
  \nonumber \\
  \label{eq:CA-x}
  &\tx_i = \frac{\tx_{i-1} x_{i+1} X_{i,i^+}}{x_i} 
          + \frac{x_{i+1}^- x_n x_{i}^+ 
             \prod_{\ell=1}^i X_{\ell^-,\ell}}{x_i},
\end{align}
for $i=2,\ldots,n-2$, then \eqref{eq:RC-xi} follows. 
The expression of $\tx_{n-1}$ and $\tx_n$ are obtained 
from \eqref{eq:C+1}--\eqref{eq:C-22}.
\end{proof}

\begin{prop}\label{prop:nearly}
We have 
$\tcR_{M,1}(Q_\ast,\bx_\ast) = (Q_\ast, \tR_{M,\ast}(\bx))$ 
for $\ast = L,C$ and $D$,
where $\tR_{M,\ast}(\bx)$ is obtained from $\tR_M(\bx)$ \eqref{eq:RonXx}, 
\eqref{eq:RonX} by setting all frozen variables in $\bx \setminus \bx_\ast$
to be $1$.
\end{prop}

\begin{proof}
As in the proof of Theorem~\ref{thm:Rcluster},
it is enough to prove that $A \circ \tR_{M,\ast} = B \circ A$.
In this proof, for simplicity we write $X_{i^+}$ and $X_{i^-}$ for 
$X_{i,i^+}$ and $X_{i^-,i}$ respectively.

Case $\ast =L$:
From Lemma~\ref{lem:x-L}, it is easy to check 
that $\tilde x_{n-1} = \tR_{M,L}(x_n)$ and $\tilde x_{n} = \tR_{M,L}(x_{n-1})$.
For $i=1,\ldots,n$ we have $\tR_{M,L}(x_i) = S' x_i$ where $S'$ is 
a Laurent polynomial of the cluster variables in $\bx_L$ independent of $i$. 
Then we obtain $A \circ \tR_{M,L} = B \circ A$,
in the same manner as Theorem~\ref{thm:Rcluster}.

Case $\ast =C$: 
For $i=n-1,n$, it is easy to check that $\tx_i = \tR_{M,C}(x_i)$
by using \eqref{eq:RC-xn-1} and \eqref{eq:RC-xn}. 
Define $\tilde A_i:= \mu_i \circ \mu_{i+1} \circ \cdots \circ \mu_{n-2}$, 
the first $n-1-i$ steps of the $n-2$ mutations in $A^{-1}$.
Let $x_i'$ be the cluster variable given by $\tx_i \mapsto x_i'$ by the 
action of $\tilde A_i$. 
In the following we show that
$\tR_{M,C}(x_i) = x_i'$ by induction on $i$, from $i=n-2$ to $1$. 
By \eqref{eq:CA-x}, for $i=1,\ldots,n-2$ we have 
\begin{align}\label{eq:RC-xevol}
  x'_i = \frac{\tx_{i-1} x_{i+1}' X_{i+1^-}}{\tx_i} 
             + \frac{x_{i+1}^- \tx_n x_{i}^+ X_{n^+} X_{1^+}\cdots X_{i-1^+}}{x'_i},
\end{align}
where we denote $x_{n-1}' = \tx_{n-1}$.
Here we have used that $\tilde A_i \circ B \circ A (Q_C) 
= \mu_{i-1} \circ \cdots \circ \mu_1(Q_C)$ and that 
the frozen variables 
$X_{i^+}$ and $X_{i+1^-}$ are interchanged by $B \circ A$.
When $i=n-2$, \eqref{eq:RC-xevol} includes $\tx_{n-1}, \tx_{n},\tx_{n-3}$. From \eqref{eq:CA-x} of $i=n-2$, \eqref{eq:RC-xn-1} and \eqref{eq:RC-xn}
respectively, we obtain the expression of $\tx_{n-3}$, $\tx_{n-1}$ and 
$\tx_{n}$ in terms of $\tx_{n-2}$. We substitute them into \eqref{eq:RC-xevol},
and show $x_{n-2}' = \tR_{M,C}(x_{n-2})$.
Now, assume that $x_{j}' = \tR_{M,C}(x_{j})$ for $j=i+1,\ldots,n-3$.
Then we have
$$
  x'_{i} \tx_i = \tx_{i-1} \tR_{M,C}(x_{i+1}) X_{i+1^-} 
             + x_{i+1}^- \tx_n x_{i}^+ X_{n^+} X_{1^+}\cdots X_{i-1^+}. 
$$ 
By using \eqref{eq:CA-x} we erase $\tx_{i-1}$ in the above, and what we have to
show is 
\begin{align*}
&\left( \tR_{M,C}(x_{i}) x_{i+1} X_{i^+} - \tR_{M,C}(x_{i+1}) x_i X_{i+1^-}
\right) \tx_i
\\
& \quad 
= x_{i+1}^- x_{i}^+ 
\left( \tR_{M,C}(x_{n}) x_{i+1} X_{n^+}X_{1^+}\cdots X_{i^+}
       - \tR_{M,C}(x_{i+1}) x_{n} X_{1^-}\cdots X_{i+1^-}
\right).
\end{align*}       
Both sides turn out to be
$$
  x_{i+1}^- x_{i}^+ \tx_i 
  \left( \prod_{\ell=1}^n X_{\ell^+}
        - \prod_{\ell=1}^n X_{\ell^-} \right),
$$
and we obtain $\tR_{M,C}(x_i) = x_i'$.

The case of $\ast =D$ follows from the case of $\ast = C$ and 
Lemma~\ref{lem:sigma-op}.
\end{proof}

To complete the proof of Theorem \ref{thm:Rcluster2}, we use the following fundamental result of Lee and Schiffler \cite{LS}, which we formulate for skew-symmetric cluster algebras of geometric type.

\begin{thm}\label{thm:LS}
Let $(x_1,x_2,\ldots,x_n,x_{n+1},\ldots,x_{n+m})$ be the cluster variables of a seed in a skew-symmetric cluster algebras of geometric type, where $x_1,\ldots,x_n$ are the mutable variables and $x_{n+1},\ldots,x_{n+m}$ are the frozen variables.  Let $y$ be any other cluster variable.  Then $y$ is a Laurent polynomial in $\Z[x_1^{\pm 1},\ldots,x_{n+m}^{\pm 1}]$ with nonnegative integer coefficients.
\end{thm}
By Theorem \ref{thm:Rcluster}, we know that $$\tcR_{M,1}(x_i)|_{\bX \to 1} = \frac{\sum_{j=1}^n x_{j+1}^- 
            \left(\prod_{\ell = j+2}^{j-1} x_{\ell} \right) x_j^+}
           {\prod_{j \neq i} x_j}$$ when all the frozen variables are set to 1.  All the $n$ (Laurent) monomials in this formula are distinct and have coefficient one.  It follows from Theorem \ref{thm:LS} that the formula for $\tcR_{M,1}(x_i)$ {\it with frozen variables} is of the form
$$\tcR_{M,1}(x_i) = \frac{\sum_{j=1}^n x_{j+1}^- 
            \left(\prod_{\ell = j+2}^{j-1} x_{\ell} \right) x_j^+ \beta_j} {\prod_{j \neq i} x_j}$$
for some Laurent monomials $\beta_j$ in the frozen variables $\bX$.  Theorem \ref{thm:Rcluster2} then follows from Proposition \ref{prop:nearly}.

\subsection{Proof of Theorem \ref{thm:Rbraid}}
\label{subsec:Rbraid}
The equality $\tR_M^2 = \id$ follows from the fact that cluster mutations are involutive.  The commutativity $\tR_{M_i} \tR_{M_j} = \tR_{M_j} \tR_{M_i}$ for $|i-j| > 1$ is clear from the definitions.

The equality $\tR_{M_j}\tR_{M_{j+1}} \tR_{M_j} = \tR_{M_{j+1}} \tR_{M_j} \tR_{M_{j+1}}$ follows from the braid relations of Theorem \ref{thm:q-YBR}, evaluated at $\e =1$.  To see this, we note that the braid relation of Theorem \ref{thm:q-YBR} can be interpreted as saying that some particular $I$-sequence $\i$ is a $\sigma$-period for a $y$-seed $(Q_{n,m},\boldy)$, where $\sigma$ is a permutation of the vertices of $Q_{n,m}$, in the sense of (any of the equivalent statements of) Theorem \ref{thm:period}.  Here, $\i$ is the sequence of mutations corresponding to $\tR_{M_j}\tR_{M_{j+1}} \tR_{M_j} \tR_{M_{j+1}} \tR_{M_j} \tR_{M_{j+1}}$.

It follows from Theorem \ref{thm:periodxy} that $\i$ is also a $\sigma$-period of $(B,\bx,\boldy)$, where $Q_{n,m} = Q(B)$ and $\boldy$ lives in the tropical semifield $\mathbb{P}_{\rm trop}(\boldy)$.  By definition, the mutation of $x$-variables in $(B,\bx,\boldy)$-seeds is exactly the mutation for the principal coefficient cluster algebra with exchange matrix $B$ \cite{CA4}.  It follows from \cite[Theorem 3.7]{CA4} that $\i$ is a $\sigma$-period for the initial seed of any skew-symmetric cluster algebra of geometric type with exchange matrix $B$, that is, for any cluster algebra obtained from adding frozen variables to $(B,\bx)$.  In particular, this applies to the cluster algebra associated to $\tQ_{n,m}$, where we consider all variables $\bX$ frozen.



\begin{thebibliography}{xx}

\bibitem[Ber]{Ber}
A.~Berenstein, Group-like elements in quantum groups and Feigin's conjecture, preprint, 1996; {\tt arXiv:q-alg/9605016}.

\bibitem[BFZ96]{BFZ}
A.~Berenstein, S.~Fomin, and A.~Zelevinsky, Parametrizations of canonical bases and totally positive matrices. Adv. Math. 122 (1996), no. 1, 49--149.

\bibitem[BFZ05]{FZ3}
A.~Berenstein, S.~Fomin and A.~Zelevinsky, Cluster algebras. III. Upper bounds and double Bruhat cells. Duke Math. J. 126 (2005), no. 1, 1--52.

\bibitem[BZ]{BZ}
A.~Berenstein and A.~Zelevinsky, Quantum cluster algebras. Adv. Math. 195 (2005), no. 2, 405--455.

\bibitem[Eti]{Et}
P.~Etingof, Geometric crystals and set-theoretical solutions to the quantum Yang-Baxter equation. Comm. Algebra 31 (2003), no. 4, 1961--1973.


\bibitem[FG09a]{FockGoncharov03}
V.~V. Fock and A.~B. Goncharov, Cluster ensembles, quantization and the dilogarithm. Ann. Sci. \'Ec. Norm. Sup\'er. (4) 42 (2009), no. 6, 865--930.

\bibitem[FG09b]{FockGonc09a}
V.~V. Fock and A.~B. Goncharov, The quantum dilogarithm and representations of quantum cluster varieties. Invent. Math. 175 (2009), no. 2, 223--286.

\bibitem[FZ02]{FominZelev02}
S.~Fomin and A.~Zelevinsky, Cluster algebras. I. Foundations. J. Amer. Math. Soc. 15 (2002), no. 2, 497--529.

\bibitem[FZ07]{CA4}
S.~Fomin and A.~Zelevinsky, Cluster algebras. IV. Coefficients. Compos. Math. 143 (2007), no. 1, 112--164. 

\bibitem[IIKKN]{IIKKN}
R.~Inoue, O.~Iyama, B.~Keller, A.~Kuniba and T.~Nakanishi,
Periodicities of T-systems and Y-systems, dilogarithm identities, and cluster algebras I: type $B_r$. Publ. Res. Inst. Math. Sci. 49 (2013), no. 1, 1--42.

\bibitem[KKMMNN]{KKMMNN} S.-J.~Kang, M.~Kashiwara,  K.C.~Misra, T.~Miwa, T.~Nakashima, and A.~Nakayashiki,
Affine crystals and vertex models. Infinite analysis, Part A, B (Kyoto, 1991), 449--484, Adv. Ser. Math. Phys., 16, World Sci. Publ., River Edge, NJ, 1992.

\bibitem[KNO]{KNO}
M.~Kashiwara,  T.~Nakashima, and M.~Okado,
Tropical R maps and affine geometric crystals. Represent. Theory 14 (2010), 446--509.

\bibitem[KNY]{KNY} 
K.~Kajiwara,  M.~Noumi, and Y.~Yamada,
Discrete Dynamical Systems with $W(A^{(1)}_{m-1} \times A^{(1)}_{n-1})$ 
Symmetry,  Lett. Math. Phys. 60 (2002), no. 3, 211--219.

\bibitem[KN]{KN}
R.~M.~Kashaev, and T.~Nakanishi,
Classical and quantum dilogarithm identities,
SIGMA Symmetry Integrability Geom. Methods Appl. 7 (2011), Paper 102, 29 pp.
 
\bibitem[LP12]{TP} 
T.~Lam and P.~Pylyavskyy,
Total positivity in loop groups, I: Whirls and curls. Adv. Math. 230 (2012), no. 3, 1222--1271.
 
\bibitem[LP13]{LP} 
T.~Lam and P.~Pylyavskyy,
Crystals and total positivity on orientable surfaces. Selecta Math. (N.S.) 19 (2013), no. 1, 173--235.

\bibitem[LPS]{LPS}
T.~Lam, P.~Pylyavskyy, and R.~Sakamoto,
Rigged Configurations and Cylindric Loop Schur Functions, preprint, 2014;
{\tt arXiv:1410.4466}.

\bibitem[LS]{LS}
K.~Lee and R.~Schiffler, Positivity for cluster algebras. Ann. of Math. (2) 182 (2015), no. 1, 73--125.


\bibitem[Nak]{Nak11}
T.~Nakanishi,
Periodicities in cluster algebras and dilogarithm identities,
{\it Representations of algebras and related topics} (A. Skowronski and K. Yamagata, eds.), EMS Series of Congress Reports, European Mathematical Society, 2011, 407--444.

\bibitem[Yam]{Yam01}
Y.~Yamada,
A birational representation of Weyl group, combinatorial R-matrix
and discerete Toda equation, {\it Physiscs and combinatorics} (A. Kirillov and N.Liskova, eds.), World Sientific Publishing, 2001, 305--319.

\end{thebibliography}
\end{document}